\documentclass[final,3p,times]{elsarticle}

\usepackage{amssymb}
\usepackage{amsmath}
\usepackage{amsfonts}

\usepackage{hyperref,amstext}
\usepackage{mathrsfs}
\usepackage{longtable}
\usepackage{booktabs}

\usepackage[title]{appendix}

\usepackage{tikz}
\usepackage[boxed]{algorithm2e}
\SetKwFor{Loop}{do}{}{loop}
\SetAlCapSkip{1em} 

\newtheorem{theorem}{Theorem}

\newtheorem{corollary}{Corollary}

\newenvironment{proof}[1][{\rm PROOF.}]{\begin{trivlist}
\item[\hskip \labelsep {\bfseries #1}]}{\end{trivlist}}

\journal{Discrete Applied Mathematics}

\begin{document}

\begin{frontmatter}



\title{Polyhedral Properties of the Induced Cluster Subgraphs}


\author{Seyedmohammadhossein Hosseinian, Sergiy Butenko}

\address{Texas A\&M University, Department of Industrial and Systems Engineering\\ College Station, TX 77843-3131, \{hosseinian, butenko\}@tamu.edu}

\begin{abstract}
A cluster graph is a graph whose every connected component is a complete graph. Given a simple undirected graph $G$, a subset of vertices inducing a cluster graph is called an independent union of cliques (IUC), and the IUC polytope associated with $G$ is defined as the convex hull of the incidence vectors of all IUCs in the graph. The {\sc Maximum IUC} problem, which is to find a maximum-cardinality IUC in a graph, finds applications in network-based data analysis. In this paper, we derive several families of facet-defining valid inequalities for the IUC polytope. We also give a complete description of this polytope for some special classes of graphs. We establish computational complexity of the separation problem for most of the considered families of valid inequalities and explore the effectiveness of employing the corresponding cutting planes in an integer (linear) programming framework for the {\sc Maximum IUC} problem through computational experiments.
\end{abstract}

\begin{keyword}
network clustering \sep cluster subgraphs \sep independent union of cliques \sep integer programming


\end{keyword}

\end{frontmatter}




\section{Introduction}
\label{sec:intro}
A graph is called a {\em cluster graph} if its every connected component is a complete graph. Given a simple undirected graph $G=(V,E)$, with the set of vertices $V$ and the set of edges $E$, an {\em independent union of cliques} (IUC) is a subset of vertices $U \subseteq V$ inducing a cluster graph. The \textsc{Maximum IUC} problem is to find an IUC of maximum cardinality in $G$. The importance of this problem stems from its application in {\em graph clustering} (partitioning the vertices of a graph into cohesive subgroups), which is a fundamental task in unsupervised data analysis; see~\cite{schaeffer2007} and the references therein for a survey of graph clustering applications and methods. A cluster graph is an ideal instance from the standpoint of cluster analysis---as it is readily composed of mutually disjoint ``tightly knit" subgroups---and the \textsc{Maximum IUC} problem aims to identify the largest induced cluster graph contained in an input graph $G$. Such an induced subgraph uncovers important information about the heterogeneous structure of the graph, hence of the data set that it represents, including the inherent number of its clusters and their cores (centers). This information is of crucial importance in network-based data clustering, as several popular graph-clustering algorithms take this information as input parameters and their performance heavily relies on its accuracy~\cite{schaeffer2007,nascimentoA2011}. Besides the practical importance of this problem, the definition of IUC subsumes two fundamental structures in a graph, i.e., cliques and independent sets. The \textsc{Maximum Clique} and \textsc{Maximum Independent Set} problems are among the most popular problems of combinatorial optimization, several variants of which have been studied in the literature; see e.g.~\cite{Karp72,BomzeA99,WuHao15,GLS,Nemhauser99, pattilloA2013}. Therefore, the study of the \textsc{Maximum IUC} problem, that admits both of these complementary structures as feasible solutions, is also of a particular theoretical interest. 

The \textsc{Maximum IUC} problem has appeared under different names in the literature. Fomin et al.~\cite{fominA2010} introduced this problem as the \textsc{Maximum Induced Cluster Subgraph} problem, and proposed an exact $\mathcal{O}(1.6181^n n^{\mathcal{O}(1)})$-time algorithm for it, where $n=|V|$. Ertem et al.~\cite{ertemA2016} considered the \textsc{Maximum IUC} problem as a relaxation of the \textsc{Maximum $\alpha$-Cluster} problem with the maximum local clustering coefficient, i.e., $\alpha = 1$. They proposed a graph clustering algorithm based on a given maximum IUC, referred to as disjoint 1-clusters in~\cite{ertemA2016}, and showed the competence of their method via experiments with real-life social networks. Even though the term ``independent union of cliques'' was used earlier (see, e.g.,~\cite{BBSB06localmax}), it was not until very recently that the corresponding optimization problem was referred to as ``\textsc{Maximum IUC}''  by Ertem et al.~\cite{ertemA2018}. In this work, the authors studied some basic properties of IUCs and analyzed complexity of the \textsc{Maximum IUC} problem on some restricted classes of graphs. They also performed computational experiments using a combinatorial algorithm (Russian Doll Search) and an integer (linear) programming formulation of this problem. To the best of our knowledge, this is the only work containing computational experiments on the \textsc{Maximum IUC} problem, and the results show that this problem is quite challenging for the exact solution methods. It is known that a graph is a cluster graph if and only if it contains no induced subgraph isomorphic to $P_3$ (the path graph on three vertices)~\cite{grammA2004}. In this regard, the \textsc{Maximum IUC} problem has also been referred to as the \textsc{Maximum Induced $P_3$-Free Subgraph} problem~\cite{fominA2010}. In our presentation to follow, we refer to induced subgraphs isomorphic to $P_3$ as {\em open triangles}. 

Evidently, finding a maximum IUC in $G$ is equivalent to finding a minimum number of vertices whose deletion  turns $G$ into a cluster graph. The latter is the optimization version of the \textsc{Cluster Vertex Deletion} problem ($k$-CVD), which asks if an input graph $G$ can be transformed into a cluster graph by deleting at most $k$ vertices, and is known to be NP-hard by Lewis and Yannakakis theorem~\cite{lewisYannakakis1980}. The study of this problem started from the viewpoint of parameterized complexity with the work of Gramm et al.~\cite{grammA2004}, who proposed an $\mathcal{O}(2.26^k + nm)$-time fixed-parameter tractable algorithm for $k$-CVD, where $n = |V|$ and $m=|E|$. Later, H{\"u}ffner et al.~\cite{huffnerA2010} and Boral et al.~\cite{boral2016} improved this result to $\mathcal{O}(2^k k^9 + nm)$ and $\mathcal{O}(1.9102^k (n+m))$, respectively. Le et al.~\cite{leA2018} showed that $k$-CVD admits a kernel with $\mathcal{O}(k^{5/3})$ vertices, which means there exists a polynomial-time algorithm that converts an $n$-vertex instance of $k$-CVD to an equivalent instance with $\mathcal{O}(k^{5/3})$ vertices. Fomin et al.~\cite{fominA2016} used the parameterized results of~\cite{boral2016} to show that the \textsc{Minimum CVD} problem, hence \textsc{Maximum IUC}, is solvable in $\mathcal{O}(1.4765^{n+o(n)})$ time, which is better than the former result of~\cite{fominA2010}. It is also known that the \textsc{Minimum CVD} problem is approximable to within a constant factor, but it is UGC-hard to approximate it within a factor strictly better than 2. You et al.~\cite{youA2017} proposed the first nontrivial approximation algorithm for \textsc{Minimum CVD} with the approximation ratio of $\frac{5}{2}$. Fiorini et al.~\cite{fioriniA2016,fioriniA2018improved} improved the approximation ratio to $\frac{7}{3}$, and later to $\frac{9}{4}$. Very recently, Aprile et al.~\cite{aprileA2020} proposed a (tight) 2-approximation algorithm for this problem; they also studied the Sherali-Adams hierarchy~\cite{Sherali90} of an integer-programming formulation of \textsc{Minimum CVD}. In particular, they have shown that, for every fixed $\epsilon >0$, the order $r = 1 + (2 / \epsilon)^4$ 
LP relaxation of this formulation (under the Sherali-Adams scheme) has the integrality gap of at most $2+\epsilon$. Finally, we should mention that several other problems involving cluster subgraphs (not necessarily induced) have been  considered in the literature as well, including  \textsc{Cluster Editing}~\cite{shamirA2004,dehneA2006,bocker2012,komusiewicz2012,bockerA2013,fominA2014,bastosA2016}, \textsc{Disjoint Cliques}~\cite{jansenA1997,AmesA2014}, \textsc{$s$-Plex Cluster Vertex Deletion}~\cite{vanA2012}, and \textsc{$s$-Plex Cluster Editing}~\cite{guoA2010}.

From a broader perspective, the \textsc{Maximum IUC} problem is to identify a maximum-cardinality {\em independent set} in a graph-based {\em independence system}. An independence system $(S,\mathcal{S})$ is a pair of a finite set $S$ together with a family $\mathcal{S}$ of subsets of $S$ such that $I_2 \subseteq I_1, I_1 \in \mathcal{S}$ implies $I_2 \in \mathcal{S}$. Referring to an independence system $(S,\mathcal{S})$, every $I \in \mathcal{S}$ is called an independent set and every $D \subseteq S, D \notin \mathcal{S}$, is called a {\em dependent set}. The minimal (inclusion-wise) dependent subsets of $S$ are called {\em circuits} of the independence system $(S,\mathcal{S})$. In graph-theory terminology, an independent set (stable set, vertex packing) is a subset of pairwise non-adjacent vertices in a graph $G=(V,E)$. Let $\mathcal{I}$ denote the set of all independent sets (in the sense of graph theory) in $G$. Then, $(V,\mathcal{I})$ is an independence system whose circuit set is given by $E$; hence, the two definitions coincide. Following this concept, $V$ together with the set of all IUCs in $G$ form an independence system whose set of circuits is given by the three-vertex subsets of $V$ inducing open triangles in the graph. This perspective on the \textsc{Maximum IUC} problem is of our interest as some of the results to follow correspond to certain properties of independence systems. It is worth noting that many well-known problems of graph theory, beyond those mentioned here, convey the notion of independence systems associated with a graph; see for example~\cite{krishnamoorthyA1979,pilipczuk2008}.  

This paper presents a comprehensive study of the \textsc{Maximum IUC} problem from the standpoint of polyhedral combinatorics. More specifically, we study the facial structure of the IUC polytope associated with a simple undirected graph $G$, and identify several facet-defining valid inequalities for this polytope. As a result, we obtain the full description of the IUC polytope for some special classes of graphs. We also study the computational complexity of the separation problems for these inequalities, and perform computational experiments to examine their effectiveness when used in a branch-and-cut scheme to solve the \textsc{Maximum IUC} problem. The organization of this paper is as follows: we continue this section by introducing the terminology and notation used in this paper. In Section~\ref{sec:polytope}, the IUC polytope is defined and the strength of its fractional counterpart---obtained form relaxing the integrality of the variables in definition of the original polytope---is discussed. In Section~\ref{sec:substructs}, we study some facet-producing structures for the IUC polytope and present the corresponding facet-inducing inequalities, as well as convex hull characterization (when possible). In our study, we pay special attention to similarities between the facial structure of the IUC polytope and those of the independent set (vertex packing) and clique polytopes. In Section~\ref{sec:separation}, we investigate the computational complexity of the separation problem for each class of the proposed valid inequalities. Section~\ref{sec:experiment} contains the results of our computational experiments, and Section~\ref{sec:conclusion} concludes this paper. 
\subsection{Terminology and notation}
\label{sec:notation}
\noindent
Throughout this paper, we consider a simple undirected graph $G=(V,E)$, where $V$ is the vertex set and $E \subseteq \{\{i,j\} \: | \: i,j \in V, i \neq j\}$ is the edge set. Unless otherwise stated, we assume $|V|=n$. Given a subset of vertices $S \subseteq V$, the subgraph {\em induced} by $S$ is denoted by $G[S]$ and is defined as $G[S]=(S,E(S))$, where $E(S)$ is the set of edges with both endpoints in $S$. A {\em clique} is a subset of vertices inducing a complete graph, and an {\em independent set} (stable set, vertex packing) is a subset of vertices inducing an edgeless graph. A clique (resp. independent set) is called {\em maximum} if it is of maximum cardinality, i.e., there is no larger clique (resp. independent set) in the graph. Cardinality of a maximum clique in $G$ is called the {\em clique number} of the graph, and is denoted by $\omega(G)$. Respectively, the {\em independence number} of $G$ is the cardinality of a maximum independent set in the graph, and is denoted by $\alpha(G)$. In a similar manner, a {\em maximum} IUC in $G$ is an IUC of maximum cardinality, and the {\em IUC number} of $G$, denoted by $\alpha^{\omega}(G)$, is the cardinality of a maximum IUC in the graph. Since every clique as well as every independent set in $G$ is an IUC, the IUC number of the graph is bounded from below by its clique number and independence number, i.e., $\max \{ \omega(G), \alpha(G) \} \leq \alpha^{\omega}(G)$. The (open) {\em neighborhood} of a vertex $i \in V$ is the set of vertices $j \in V \backslash \{i\}$ adjacent to $i$, and is denoted by $N(i)$. The {\em closed neighborhood} of $i$ is defined as $N[i]=N(i) \cup \{i\}$. The {\em incidence vector} of a subset of vertices $S \subseteq V$ is a binary vector $x \in \mathbb{R}^n$ such that $x_i=1,\forall i \in S,$ and $x_i=0,\forall i \in V \backslash S$. Conventionally, the incidence vector of a single vertex $i \in V$ is denoted by $e_i$. We use ${\bf 0}$, ${\bf 1}$ and ${\bf 2}$ to denote the vectors of all zero, one and two, respectively. The corresponding dimensions will be clear from the context.

We assume familiarity of the reader with fundamentals of polyhedral theory, and just briefly mention the concept of {\em lifting}~\cite{Padberg75,wolsey76,zemel78}, as several proofs to follow are based on lifting arguments. Let $\mathscr{P}$ be the convex hull of the set of 0-1 feasible solutions to an arbitrary system of linear inequalities with a nonnegative coefficient matrix, i.e., $\mathscr{P}= \text{conv} \{x \in \{0,1\}^n \: | \: Ax \leq b \}$. Corresponding to a subset of variables indexed by $N' \subseteq N = \{1,\ldots,n\}$, consider the polytope $\mathscr{P}_{(x_{N'}=0)}$ obtained from $\mathscr{P}$ by setting $x_j=0, \forall j \in N'$. Let $N''=N \backslash N'$, and suppose that the inequality $\sum_{i \in N''} \pi_i x_i \leq \pi_0$ induces a facet of $\mathscr{P}_{(x_{N'}=0)}$. Then, lifting a variable $x_j, j \in N'$, into this inequality leads to an inequality $\pi_j x_j + \sum_{i \in N''} \pi_i x_i \leq \pi_0$, where
\begin{equation} \label{eq:lift}
\pi_j =\pi_0 - \max \left \{ \sum_{i \in N''} \pi_i x_i | \sum_{i \in N''} A_i x_i \leq b - A_j ; \: x_i \in \{0,1\}, \forall i \in N'' \right \},
\end{equation}
which is facet-defining for $\mathscr{P}_{(x_{N' \backslash \{j\}}=0)}$. In~\eqref{eq:lift}, $A_i$ denotes the coefficient vector of a variable $x_i$ in the system of inequalities, i.e., the $i$-th column of $A$. The importance of this result is due to the fact that some facets of the original polytope $\mathscr{P}$ can be obtained from the facets of the restricted polytopes through sequential (or simultaneous) lifting procedures. In particular, some facets of the polytope associated with a graph $G$ for a certain problem may be identified through lifting the facets of the lower-dimensional polytopes corresponding to its subgraphs. 
Finally, note that~\eqref{eq:lift} is equivalent to 
\begin{equation}
\pi_j = \pi_0 - \max \left \{ \sum_{i \in N''} \pi_i x_i \: | \: x \in \mathscr{P} \text{ and } x_j=1 \right \},
\end{equation}
which is easier to use when $\mathscr{P}$ is defined to be the convex hull of the incidence vectors of the sets of vertices in a graph holding a certain property.
\section{The IUC Polytope}
\label{sec:polytope}
\noindent
The IUC polytope associated with a simple, undirected graph $G$, denoted by $\mathscr{P}_{_{\!\!\mathcal{IUC}}}(G)$, is the convex hull of the incidence vectors of all IUCs in $G$, which can be characterized by the circuit set of the corresponding independence system as follows~\cite{NemhauserTrotter74}:
\begin{equation}
\mathscr{P}_{_{\!\!\mathcal{IUC}}}(G) = \text{conv} \{x \in \{0,1\}^n \: | \: x_i + x_j + x_k \leq 2, \forall \{i,j,k\} \in \Lambda(G)\},
\end{equation}
where $\Lambda(G)$ denotes the set of three-vertex subsets of $V$ inducing open triangles in $G$. Patently, $\alpha^\omega (G) = \max \{\sum_{i \in V} x_i \: | \: x \in \mathscr{P}_{_{\!\!\mathcal{IUC}}}(G) \}$, and the {\em maximum weight IUC} problem, which is to find an IUC with the maximum total weight in a vertex-weighted graph, can be formulated as 
\begin{equation}
\text{Maximize } \sum_{i \in V} w_i x_i \text{ ~subject to~ } x \in \mathscr{P}_{_{\!\!\mathcal{IUC}}}(G),
\end{equation}
with $w_i$ denoting the weight of a vertex $i \in V$. We call each inequality $x_i + x_j + x_k \leq 2, \{i,j,k\} \in \Lambda(G)$, an open-triangle (OT) inequality. The following theorems establish the basic properties of $\mathscr{P}_{_{\!\!\mathcal{IUC}}}(G)$ including the strength of the OT inequalities.
\begin{theorem}
$\mathscr{P}_{_{\!\!\mathcal{IUC}}}(G)$ is full-dimensional, and for every $i \in V$, the inequalities $x_i \geq 0$ and $x_i \leq 1$ are facet-defining for this polytope.
\end{theorem}
\begin{proof}
Note that ${\bf 0} \in \mathscr{P}_{_{\!\!\mathcal{IUC}}}(G)$. Also, every single vertex, as well as every pair of vertices in $G$ is an IUC, by definition.
Full-dimensionality of $\mathscr{P}_{_{\!\!\mathcal{IUC}}}(G)$ is established by noting that ${\bf 0}$ and $e_i, \forall i \in V$, are $n+1$ affinely independent points in $\mathbb{R}^n$ belonging to this polytope. Besides, for every $i \in V$, {\bf 0} and $e_j, \forall j \in V \backslash \{i\}$, belong to $\mathcal{F}^0_i = \{x \in \mathscr{P}_{_{\!\!\mathcal{IUC}}}(G) \: | \: x_i = 0\}$, which indicates that $\mathcal{F}^0_i$ is a facet of $\mathscr{P}_{_{\!\!\mathcal{IUC}}}(G)$. Similarly, for every $i \in V$, the points $e_i$ and $e_{ij} = e_i + e_j, \forall j \in V \backslash \{i\}$, are affinely independent and belong to $\mathcal{F}^1_i = \{x \in \mathscr{P}_{_{\!\!\mathcal{IUC}}}(G) \: | \: x_i = 1\}$, hence $\mathcal{F}^1_i$ is a facet of $\mathscr{P}_{_{\!\!\mathcal{IUC}}}(G)$.
\end{proof}

It follows from full-dimensionality of $\mathscr{P}_{_{\!\!\mathcal{IUC}}}(G)$ that every facet-defining inequality of this polytope is unique to within a positive multiple. Thus, in order to prove that a proper face $\mathcal{F}$ of $\mathscr{P}_{_{\!\!\mathcal{IUC}}}(G)$ defined by a valid inequality $\pi(x) \leq \pi_0$ is a facet, it suffices to show that any supporting hyperplane of $\mathscr{P}_{_{\!\!\mathcal{IUC}}}(G)$ containing $\mathcal{F}$ is a (non-zero) scalar multiple of $\pi(x) = \pi_0$.

\begin{theorem} \label{th:p3}
The inequality $x_i + x_j +x_k \leq 2, \{i,j,k\} \in \Lambda(G)$, is facet-defining for $\mathscr{P}_{_{\!\!\mathcal{IUC}}}(G)$ if and only if there does not exist a vertex $v \in V\backslash \{i,j,k\}$ such that $H = \{i,j,k,v\}$ induces a chordless cycle in $G$.
\end{theorem}
\begin{proof}
Clearly, $\mathcal{F} = \{x \in \mathscr{P}_{_{\!\!\mathcal{IUC}}}(G) \: | \: x_i + x_j +x_k = 2\}$ is a proper face of $\mathscr{P}_{_{\!\!\mathcal{IUC}}}(G)$. So, suppose that it is contained in a supporting hyperplane $\mathscr{H}: \sum_{v \in V} a_v x_v = b$ of this polytope. Observe that,
$$
e_i + e_j \in \mathcal{F} \subseteq \mathscr{H} \: \: \therefore \: \: a_i + a_j = b.
$$
Similarly, $a_i + a_k = b$ and $a_j + a_k = b$, which imply $a_i = a_j = a_k = a$ and $b= 2a$. Now, consider $G[H]$ the subgraph induced by $H = \{i,j,k,v\}$ for an arbitrary vertex $v \in V\backslash \{i,j,k\}$. As $G[H]$ is not a chordless cycle, it contains at most three open triangles, and those have at least one vertex in common. Let $k$ be such a vertex, then
$$
e_i + e_j + e_v \in \mathcal{F} \subseteq \mathscr{H} \: \: \therefore \: \: a_i + a_j + a_v = b \: \: \therefore \: \: a_v = 0,~\forall v \in V\backslash \{i,j,k\}. 
$$
As the interior of $\mathscr{P}_{_{\!\!\mathcal{IUC}}}(G)$ is nonempty, $a \neq 0$ and $\mathscr{H}$ is a (non-zero) scalar multiple of $x_i + x_j +x_k = 2$. To complete the proof, note that if $H = \{i,j,k,v\}$ induces a chordless cycle in $G$, then $x_i + x_j +x_k \leq 2$ is dominated by the inequality $x_i + x_j +x_k+x_v \leq 2$, which is facet-defining for $\mathscr{P}_{_{\!\!\mathcal{IUC}}}(G)$ by Theorem~\ref{th:hole} of the next section. 
\end{proof}

In Section~\ref{sec:substructs} we will show that Theorem~\ref{th:p3} is a special case of a more general statement for the so-called \emph{star inequality}, defined later. 

Obviously, the {maximum IUC} problem can be directly solved using off-the-shelf MIP solvers through the following formulation:
\begin{equation} \label{eq:IPiuc}
\begin{aligned}
\alpha^\omega (G) \: = \: \underset{}{\max} \quad
& \sum_{i \in V} x_i \\
\text{s.t.} \quad 
& x_i + x_j + x_k \leq 2, && \forall \{i,j,k\} \in \Lambda(G),\\
& x_i \in \{0,1\}, && \forall \: i \in V.
\end{aligned}
\end{equation}
However, the LP relaxation of~\eqref{eq:IPiuc} is generally too weak. In fact, given the feasibility of the fractional assignment $x_i=\frac{2}{3},\forall i \in V$, the optimal solution value of the LP relaxation problem is at least as large as $\frac{2n}{3}$, while the computational results of Ertem et al.~\cite{ertemA2018} show that the IUC number of graphs with moderate densities tends to stay in a close range of their relatively small independence and clique numbers. This is due to the fact that the fractional IUC polytope $\mathscr{P}^{F}_{_{\!\!\mathcal{IUC}}}(G) = \{x \in [0,1]^n \: | \: x_i + x_j + x_k \leq 2, \forall \{i,j,k\} \in \Lambda(G)\}$ is a polyhedral relaxation of a {\em cubically} constrained region in the space of original variables. Note that the IUC property may be enforced through trilinear products of the variables in a 0-1 program, i.e., by replacing the OT inequalities of~\eqref{eq:IPiuc} with $x_i x_j x_k =0, \forall \{i,j,k\} \in \Lambda(G)$. In such a cubic formulation, relaxing the integrality of variables will not change the set of optimal solutions, because given a set of variables $x_i=0,\forall i \in V_0 \subset V$, to satisfy the constraints, the optimality conditions necessitate $x_i=1,\forall i \in V \backslash V_0$. Hence, $\alpha^\omega(G)$ can be obtained by solving the following continuous problem:  
\begin{equation} \label{eq:IPpoly}
\alpha^\omega (G) \: = \: \underset{x \in [0,1]^n}{\max} \left \{ \sum_{i \in V} x_i \: | \: x_i x_j x_k =0, \forall \{i,j,k\} \in \Lambda(G)\right \}.
\end{equation}
It is known that the (convex and concave) envelopes of a trilinear monomial term $w=x y z$ when $0 \leq x,y,z \leq 1$ are as follows~\cite{meyerA2004}:
$$
w \leq x, \quad w \leq y, \quad w \leq z, \quad w \geq x+y+z-2, \quad w \geq 0, 
$$
which, along with $w_{ijk}=x_i x_j x_k =0, \forall \{i,j,k\} \in \Lambda(G)$, indicates that $\mathscr{P}_{_{\!\!\mathcal{IUC}}}^{F}(G)$ is a polyhedral outer-approximation of the feasible region of~\eqref{eq:IPpoly}. As a result, the gap between $\alpha^\omega(G)$ and the optimal solution value of the LP relaxation problem of~\eqref{eq:IPiuc} is due to accumulation of the gap between the trilinear terms and their overestimating envelopes. A similar result holds between continuous and integer (linear) programming formulations of the {maximum independent set} and {maximum clique} problems, except that the corresponding continuous formulations of those problems are quadratically constrained. The cubic nature of the IUC formulation justifies the weakness of $\mathscr{P}^{F}_{_{\!\!\mathcal{IUC}}}(G)$ even compared to the fractional independent set (vertex packing) and clique polytopes, which are known to be loose in general. This further reveals the importance of exploring the facial structure of the IUC polytope and applying strong cutting planes in integer (linear) programming solution methods of the {maximum IUC} problem.
\section{Facet-producing Structures}
\label{sec:substructs}
\noindent
In this section, we present some facet-inducing valid inequalities for the IUC polytope associated with chordless cycle (and its edge complement), star (and double-star), fan and wheel graphs. We study the conditions under which these inequalities remain facet-defining for the IUC polytope of an arbitrary graph that contains the corresponding structure as an induced subgraph, and present some results with respect to lifting procedure for those that do not have this property. In some cases, we will present the full description of the IUC polytope associated with these graphs.
\subsection{Chordless cycle and its complement}
\label{sec:cycle}
\noindent
The following theorem concerns the facial structure of the IUC polytope associated with a (chordless) cycle graph. In an arbitrary graph $G=(V,E)$, a subset of vertices $H \subseteq V,~|H| \geq 4$, inducing a chordless cycle is called a {\em hole}. In labeling the vertices of a hole $H$, we assume that adjacent vertices have consecutive labels with the convention $|H|+1 \equiv 1$.
\begin{theorem} \label{th:hole} \textbf{\em (Hole Inequality)} 
Let $H \subseteq V$ be a hole of cardinality $|H|=3q+r$ in $G=(V,E)$, where $q$ is a positive integer and $r \in \{0,1,2\}$. Then, the inequality
\begin{equation} \label{eq:hole}
\sum_{i \in H} x_i \leq 2q + \left \lfloor \frac{2r}{3} \right \rfloor
\end{equation}
is valid for $\mathscr{P}_{_{\!\!\mathcal{IUC}}}(G)$. Moreover, {\em (i)} it induces a facet of $\mathscr{P}_{_{\!\!\mathcal{IUC}}}(G[H])$ if and only if $r \neq 0$, and {\em (ii)} it is facet-defining for $\mathscr{P}_{_{\!\!\mathcal{IUC}}}(G)$ if $q=r=1$.
\end{theorem}
\begin{proof}
Consider $G[H]$ the subgraph induced by $H$. This subgraph contains $|H|$ distinct open triangles, and each vertex $i \in H$ appears in exactly three of them. Summation of the corresponding OT inequalities leads to
$$
3 \sum_{i \in H} x_i = \sum_{\{i,j,k\} \in \Lambda(G[H])} x_i + x_j +x_k \leq 2 |H| = 6q + 2r.
$$
After dividing both sides of the last inequality by $3$, integrality of the left-hand side implies validity of~\eqref{eq:hole} for $\mathscr{P}_{_{\!\!\mathcal{IUC}}}(G)$ as a Gomory-Chv\'{a}tal cut. 

(i) It is clear that, as a linear combination of OT inequalities,~\eqref{eq:hole} is not facet-defining for $\mathscr{P}_{_{\!\!\mathcal{IUC}}}(G[H])$ if $r=0$. Thus, we just need to show that $r \neq 0$ is sufficient for~\eqref{eq:hole} to induce a facet of this polytope. We examine the corresponding cases separately.  
\begin{itemize}
\item \textsc{case 1}: $r=1$\\
Consider a sequential partitioning of $H$ such that each partition contains exactly three vertices except for the last one, which is left with a single vertex. Assume a labeling of the vertices such that the first partition is given by $\{1,2,3\}$, and let $x_{U_1}$ and $x_{U_2}$ be the incidence vectors of the following sets of vertices:
$$
\begin{aligned}
U_1 &= \{1,3,3p-1,3p,\forall p \in \{2,\ldots,q\} \},\\
U_2 &= \{2,3,3p-1,3p,\forall p \in \{2,\ldots,q\} \}.
\end{aligned}
$$
The sets of black vertices in Figures~\ref{fig:hole_1}$(a)$ and~\ref{fig:hole_1}$(b)$ depict $U_1$ and $U_2$, respectively.

Observe that, $x_{U_1}, x_{U_2} \in \mathcal{F} = \{x \in \mathscr{P}_{_{\!\!\mathcal{IUC}}}(G[H]) \: | \: \sum_{v \in H} x_v = 2q\}$, hence $\mathcal{F}$ is a proper face of $\mathscr{P}_{_{\!\!\mathcal{IUC}}}(G[H])$. Suppose that $\mathscr{H}:\sum_{v \in H} a_v x_v = b$ is an arbitrary supporting hyperplane of $\mathscr{P}_{_{\!\!\mathcal{IUC}}}(G[H])$ containing $\mathcal{F}$. Then, $x_{U_1}, x_{U_2} \in \mathscr{H}$ implies that $a_1 = a_2$. The same argument using similar partitions can be used to show that this result holds for every two adjacent vertices in $G[H]$, thus $a_v = a, \forall v \in H$, and $b=2qa$. The proof is complete by noting that the interior of $\mathscr{P}_{_{\!\!\mathcal{IUC}}}(G[H])$ is nonempty, so $a \neq 0$.
\item \textsc{case 2}: $r=2$\\
Similar to the former case, let $x_{U_1}$ and $x_{U_2}$ be the incidence vectors of 
$$
\begin{aligned}
U_1 &= \{3q+2,1,3,3p-1,3p,\forall p \in \{2,\ldots,q\} \},\\
U_2 &= \{3q+2,2,3,3p-1,3p,\forall p \in \{2,\ldots,q\} \}.
\end{aligned}
$$
%
Then, the same argument leads to the result that the inequality $\sum_{v \in H} x_v \leq 2q + 1$ induces a facet of $\mathscr{P}_{_{\!\!\mathcal{IUC}}}(G[H])$. 
%
%
\end{itemize}

\begin{figure}[!t] 
\centering{
\includegraphics[width=30em]{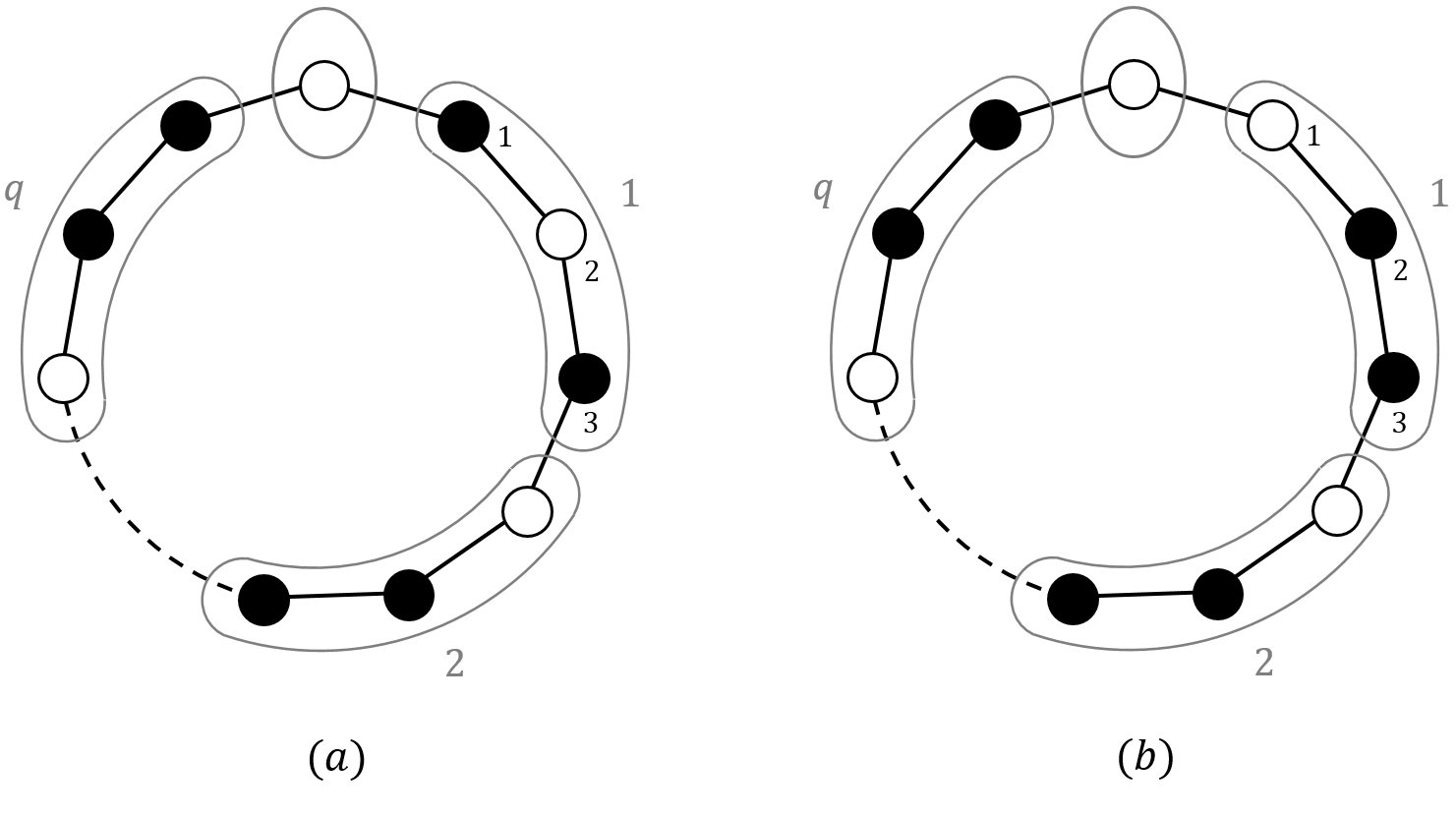}
}
\caption{Subgraph induced by a hole $H$, $|H| =3q+1$.}
\label{fig:hole_1}
\end{figure}

(ii) To conclude the proof, we show that~\eqref{eq:hole} is facet-defining for $\mathscr{P}_{_{\!\!\mathcal{IUC}}}(G)$ if $q=r=1$, i.e., $|H|=4$. By (i), $\sum_{i \in H} x_i \leq 2$ defines a facet of $\mathscr{P}_{_{\!\!\mathcal{IUC}}}(G[H])$. Let $\sum_{i \in H} x_i + \sum_{v \in V \backslash H} a_v x_v \leq 2$ be the inequality obtained from an arbitrary sequential lifting of $\sum_{i \in H} x_i \leq 2$ (to $\mathbb{R}^n$). Observe that, for every vertex $v \in V \backslash H$, there always exists two vertices $i,j \in H$ such that $\{i,j,v\} \notin \Lambda(G)$, thus, $\max \{ \sum_{i \in H} x_i \: | \: x \in \mathscr{P}_{_{\!\!\mathcal{IUC}}}(G) \text{ and } x_v = 1\} = 2$. This implies $a_v = 0, \forall v \in V \backslash H$, hence $\sum_{i \in H} x_i \leq 2$ induces a facet of $\mathscr{P}_{_{\!\!\mathcal{IUC}}}(G)$.
\end{proof}

It is easy to see the similarity between~\eqref{eq:hole} and the {\em odd-hole} inequality for the vertex packing (independent set in the sense of graph theory) polytope associated with $G$. Padberg~\cite{Padberg73} proved that, given a hole $H$ in $G$ with an odd number of vertices, the inequality $\sum_{i \in H} x_i \leq \frac{1}{2} (|H| - 1)$ is facet-defining for the vertex packing polytope associated with $G[H]$. Later, Nemhauser and Trotter~\cite{NemhauserTrotter74} noted that the odd-hole inequality is a special case of a more general result concerning independence systems, stated as follows:
\begin{theorem} \label{th:NT74}
{\em \cite{NemhauserTrotter74}} Suppose $\alpha_0$ is the cardinality of a maximum independent set in an independence system $\mathscr{S}=(S,\mathcal{S})$, and $\mathscr{S}$ contains $n=|S|$ maximum independent sets $I_1,I_2,\ldots,I_n$ with corresponding affinely independent (incidence) vectors $x^1,x^2,\ldots,x^n$. Then $\sum_{i \in S} x_i \leq \alpha_0$ is facet-defining for the polytope associated with $\mathscr{S}$, i.e., the convex hull of the incidence vectors of all members of $\mathcal{S}$ in $\mathbb{R}^n$.
\end{theorem}

Similar to the odd-hole inequality for the vertex packing polytope, Theorem~\ref{th:hole}(i) spots a facet-inducing inequality implied by Theorem~\ref{th:NT74} for the independence system defined by the IUC property on $G[H]$. Note that Theorem~\ref{th:NT74} itself is of little practical value as it provides no means to identify the independence systems satisfying such properties. 

Furthermore, Theorems~\ref{th:p3} and~\ref{th:hole}(ii) correspond to {\em maximal clique} inequality of independence systems. An independence system $(S,\mathcal{S})$ is called $k$-regular if every circuit of it is of cardinality $k$. Referring to a $k$-regular independence system $(S,\mathcal{S})$, $C \subseteq S$ is called a {\em clique} if $|C| \geq k$ and all $\binom{|C|}{k}$ subsets of $C$ are circuits of $(S,\mathcal{S})$. Clearly, this definition coincides with the graph-theoretic definition of a clique. Nemhauser and Trotter~\cite{NemhauserTrotter74} also proved that the inequality $\sum_{i \in C} x_i \leq k-1$ is facet-inducing for the polytope associated with a $k$-regular independence system $(S,\mathcal{S})$ if $C \subseteq S$ is a maximal (inclusion-wise) clique. A special case of this result for the vertex packing polytope associated with a graph $G$ had been originally shown by Padberg~\cite{Padberg73}. Let $\mathcal{U}$ denote the set of all IUCs in $G$, and $(V,\mathcal{U})$ be the corresponding independence system. Patently, every $\{i,j,k\} \in \Lambda(G)$ is a clique in $(V,\mathcal{U})$, and such a clique is maximal only if it is not contained in a clique of cardinality 4. It is not hard to see that a clique of size 4 in $(V,\mathcal{U})$ must be a hole in $G$. Therefore, Theorem~\ref{th:p3} corresponds to maximality of $\{i,j,k\} \in \Lambda(G)$ as a clique of the corresponding independence system. Besides, by the proof of Theorem~\ref{th:hole}(ii), it became clear that $(V,\mathcal{U})$ cannot have a clique of cardinality greater than 4 because, for every 4-hole $H$ and every vertex $v \in V \backslash H$, there always exist two vertices $i,j \in H$ such that $\{i,j,v\} \notin \Lambda(G)$. Thus, a hole of cardinality 4 in $G$ is actually a maximum clique in $(V,\mathcal{U})$, and Theorem~\ref{th:hole}(ii) corresponds to its maximality. This result is of a special value for the {maximum IUC} problem as all facets of this type can be identified in polynomial time. Recall that the vertex packing polytope associated with $G$ may possess an exponential number of such facets.

Given this result, we extend the concept of {\em fractional clique/independence number} of a graph $G$ to the independence system $(V,\mathcal{U})$. The fractional clique number of $G$ is the maximum sum of nonnegative weights that can be assigned to its vertices such that the total weight of every independent set in the graph is at most 1. Naturally, the fractional clique number of $G$ is equivalent to the fractional independence number of its complement graph, where the total weight of every clique is at most 1. We define the {\em fractional IUC number} of a graph $G$ as the maximum sum of nonnegative (and no more than 1) weights that can be assigned to its vertices such that the total weight of every clique in $(V,\mathcal{U})$, i.e., the vertex set of every open triangle and 4-hole in $G$, is at most 2. Clearly, the fractional IUC number of $G$ provides an upper bound on its IUC number, which is computable in polynomial time as opposed to the fractional clique and independence numbers of a graph which are NP-hard to compute~\cite{GLS}.

The vertex packing polytope associated with a cycle graph is characterized by the set of maximal clique inequalities, i.e., $x_i + x_j \leq 1,~\forall \{i,j\} \in E$, the variable (lower) bounds, and the Padberg's odd-hole inequality if the number of vertices is odd. However, the IUC polytope associated with a cycle graph has more facets than those defined by these families. As a case in point, consider a cycle graph on $n = 12$ vertices, and observe that the point given by $x_i = 1,~\forall i \in \{1,5,9\}$, and $x_i = \frac{1}{2},~\forall i \in V \backslash \{1,5,9\}$, is an extreme point of the polytope defined by the corresponding set of OT inequalities and the variable bounds. In fact, a fractional extreme point of a similar structure---generated by OT inequalities and variable bounds---exists for every cycle graph on $n = 3 q_1 + 4 q_2$ vertices, for some nonnegative integers $q_1$ and $q_2$, if $q_2$ is odd. In such a cycle graph, fixing $q_1 + q_2$ variables whose vertices are located at distances 3 or 4 from each other at 1 turns the set of OT inequalities into the set of vertex packing inequalities associated with a conflict cycle graph with an odd number of vertices; this leads to extremity of a point with half integral values on the remaining $n - (q_1 + q_2)$ variables. It is easy to check that, unless $q_2 = 1$, this fractional point always satisfies the hole inequality~\eqref{eq:hole}.

Given the definition of IUC, it is also natural to consider the correspondence between the IUC polytope and the clique polytope associated with a graph $G$, i.e., the convex hull of the incidence vectors of all cliques in $G$. Patently, the odd-hole inequality for the vertex packing polytope translates to the {\em odd-anti-hole} inequality for the clique polytope associated with $G$. An {\em anti-hole} is a subset of vertices $A \subseteq V,~|A| \geq 4$, such that the (edge) complement of the corresponding induced subgraph is a chordless cycle, i.e., $A$ is a hole in the complement graph of $G$. The following theorem states that, under a slightly stronger condition, the same inequality is valid for the IUC polytope associated with $G$ and possesses the same facial property on the polytope associated with the corresponding induced subgraph.
\begin{theorem} \label{th:antihole} \textbf{\em (Anti-hole Inequality)}
Let $A \subseteq V,~|A| \geq 6$, be an anti-hole in $G=(V,E)$. Then, the inequality
\begin{equation} \label{eq:antihole}
\sum_{i \in A} x_i \leq \left \lfloor \frac{|A|}{2} \right \rfloor
\end{equation}
is valid for $\mathscr{P}_{_{\!\!\mathcal{IUC}}}(G)$. Furthermore, it induces a facet of $\mathscr{P}_{_{\!\!\mathcal{IUC}}}(G[A])$ if and only if $|A|$ is odd.
\end{theorem}
\begin{proof}
Consider a labeling of the vertices of an anti-hole $A,~|A| \geq 6$, such that $\{i,i+1\}, \forall i \in \{1,\ldots,|A|-1\}$, and $\{|A|,1\}$ identify the pairs of non-adjacent vertices in $G[A]$. For a fixed sequence of labels, note that $\{1,2,i,i+1\}, \forall i \in \{4,\ldots,|A|-2\}$, is a hole of cardinality 4, thus the inequality $x_1 + x_2 + x_i + x_{i+1} \leq 2, \forall i \in \{4,\ldots,|A|-2\}$, is valid for $\mathscr{P}_{_{\!\!\mathcal{IUC}}}(G)$. Consider the summation of those $|A|-5$ valid inequalities. Observe that each vertex $j \in \{1,2\}$ appears $|A|-5$ times, $j \in \{4,|A|-1\}$ once, and $j \in \{5,\ldots,|A|-2\}$ twice, in the left-hand side of the resultant inequality. These partitions are illustrated by black, hatched, and gray vertices in Figure~\ref{fig:antihole}, respectively. The white vertices in Figure~\ref{fig:antihole} are those that are not present in this inequality. 
Since the starting vertex of the labeling sequence was arbitrary, $|A|$ valid inequalities of this type may be written for $\mathscr{P}_{_{\!\!\mathcal{IUC}}}(G)$. Considering all those inequalities, every vertex will appear exactly twice in a black position, twice in a white, twice in a hatched, and $|A|-6$ times in a gray position. As a result, the summation of all 4-hole inequalities associated with $G[A]$ results in 
$$
\Big ( 2(|A|-5) + 2 + 2(|A|-6) \Big ) \sum_{i \in A} x_i \leq 2 |A| (|A|-5).
$$
\begin{figure}[!t]
\centering{
\includegraphics[width=17em]{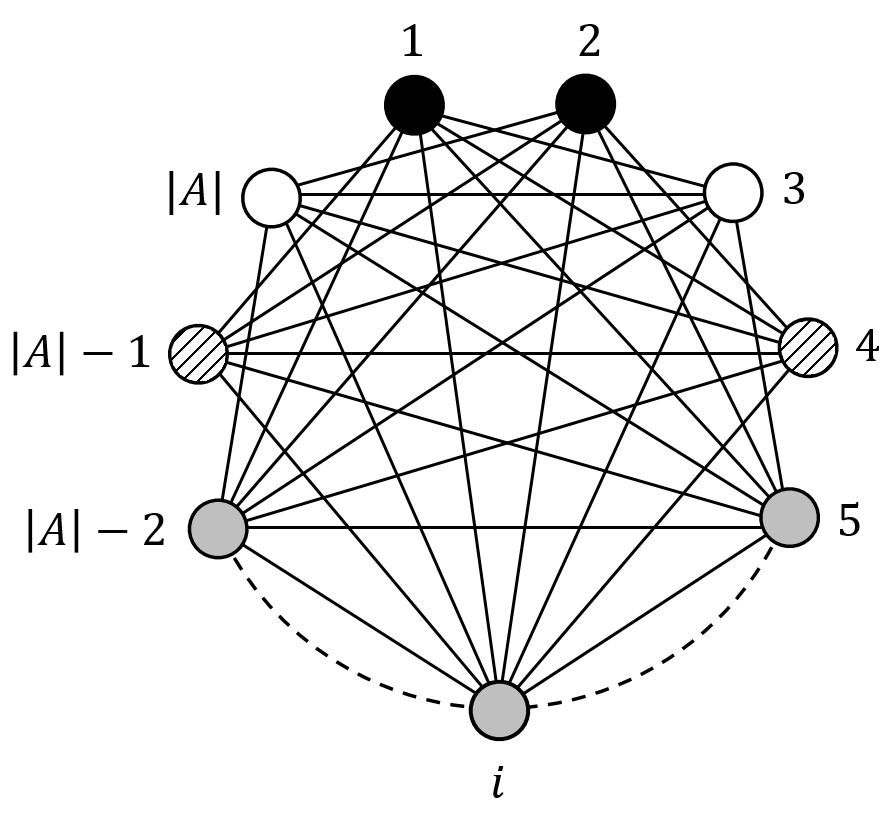}
}
\caption{Subgraph induced by an anti-hole.}
\label{fig:antihole}
\end{figure}
Therefore,
$
\sum_{i \in A} x_i \leq \frac{1}{2}|A|
$
is valid for $\mathscr{P}_{_{\!\!\mathcal{IUC}}}(G)$, which further implies validity of~\eqref{eq:antihole} as the Gomory-Chv\'{a}tal cut corresponding to this inequality. Clearly, when $|A|$ is even,~\eqref{eq:antihole} is a linear combination of the 4-hole inequalities associated with $G[A]$, hence not facet-defining for the corresponding polytope. Thus, it remains to show that it induces a facet of $\mathscr{P}_{_{\!\!\mathcal{IUC}}}(G[A])$ if $|A|$ is odd.

Given an odd anti-hole $A$, consider $\mathcal{F} = \{x \in \mathscr{P}_{_{\!\!\mathcal{IUC}}}(G[A]) \: | \: \sum_{v \in A} x_v = \frac{1}{2}(|A|-1)\}$. Associated with a fixed labeling of $A$ as described above, let $x_C$ be the incidence vector of the set of even vertices, i.e., $C = \{2,4,\ldots,|A|-1\}$. Observe that $C$ is a clique in $G[A]$ and $x_C \in \mathcal{F}$, which implies $\mathcal{F}$ is a proper face of $\mathscr{P}_{_{\!\!\mathcal{IUC}}}(G[A])$. Suppose that $\mathcal{F}$ is contained in a supporting hyperplane $\mathscr{H}:\sum_{i \in A} a_i x_i = b$ of $\mathscr{P}_{_{\!\!\mathcal{IUC}}}(G[A])$. Let $C' = (C \cup \{1\}) \backslash \{2\}$ and note that $C'$ is also a clique, hence $x_{C'} \in \mathcal{F}$. Then, $x_C,x_{C'} \in \mathscr{H}$ implies that $a_1 = a_2$. Since the start vertex of the labeling was arbitrary, this results holds for all pairs of vertices with consecutive labels. That is, $a_i = a, \forall i \in A$, and $b=\frac{1}{2}a(|A|-1)$. The proof is complete noting that the interior of $\mathscr{P}_{_{\!\!\mathcal{IUC}}}(G[A])$ is nonempty, thus $a \neq 0$. 
\end{proof}

Theorem~\ref{th:antihole} is closely related to a basic property of IUCs. Ertem et al.~\cite{ertemA2018} showed that, in an arbitrary graph, a (maximal) clique is a maximal IUC if and only if it is a dominating set. In an anti-cycle graph $G=(A,E)$, each maximal clique is maximum, and in case $|A| \geq 6$, it is also a maximum IUC by Theorem~\ref{th:antihole}. It is easy to verify that the anti-cycle graphs on 4 and 5 vertices are the only ones in which a maximal clique is not a dominating set.

In the subsequent sections, we will show that the IUC polytope associated with an anti-cycle graph has many other facets than those induced by the 4-hole inequalities and~\eqref{eq:antihole}, due to the {\em fan} substructures that appear in this graph. We postpone our discussion in this regard to Section~\ref{sec:fanwheel}. 

Finally, it should be mentioned that the odd-hole and odd-anti-hole inequalities for the vertex packing polytope have been generalized by Trotter~\cite{Trotter75} via introducing a general class of facet-producing subgraphs, called {\em webs}, that subsumes cycle and anti-cycle graphs as special cases. Given the connection between the vertex packing and IUC polytopes, webs may lead to further facet-inducing inequalities for the IUC polytope as well. We leave this topic for future research and will not pursue it in this paper. 
\subsubsection*{Lifting hole and anti-hole inequalities}
\label{sec:cyclelift}
\noindent
Let $G[H]$,~$|H| \geq 5$, be an induced chordless cycle in $G=(V,E)$, where $|H|=3q + r$ and $r \neq 0$. Then by Theorem~\ref{th:hole}, inequality~\eqref{eq:hole} induces a facet of $\mathscr{P}_{_{\!\!\mathcal{IUC}}}(G[H])$, but it is not necessarily facet-defining for $\mathscr{P}_{_{\!\!\mathcal{IUC}}}(G)$. Consider lifting a variable $x_v,~v \in V \backslash H$, into~\eqref{eq:hole} to generate a facet-defining inequality for $\mathscr{P}_{_{\!\!\mathcal{IUC}}}(G[H \cup \{v\}])$ of the form $a_v x_v + \sum_{i \in H} x_i \leq 2q + \left \lfloor \frac{2r}{3} \right \rfloor$. The value of $a_v$ in such an inequality depends on cardinality of $H \cap N(v)$ as well as distances (lengths of the shortest paths) among the vertices of $N(v)$ in $G[H]$. As a case in point, consider the graphs of Figure~\ref{fig:holelift}, and lifting $x_8$ into the inequality $\sum_{i=1}^7 x_i \leq 4$ corresponding to $H=\{1, \ldots, 7\}$. In both graphs, $|H \cap N(8)|=4$ and the difference is due to the distance between the vertices 4 and 7 (equivalently, 1 and 7) in the left-hand-side graph, and the vertices 4 and 6 (equivalently, 1 and 6) in the right-hand-side graph. The sets of black vertices illustrate the maximum size IUCs containing vertex 8 in these graphs. Clearly, the coefficient of $x_8$ vanishes in lifting $\sum_{i=1}^7 x_i \leq 4$ for the left-hand-side graph, indicating that this inequality is facet-defining for $\mathscr{P}_{_{\!\!\mathcal{IUC}}}(G[H \cup \{8\}])$. On the other hand, for the IUC polytope associated with the right-hand-side graph, this inequality is dominated by the lifted (facet-defining) inequality $x_8 + \sum_{i=1}^7 x_i \leq 4$. 
\begin{figure}[!t] 
\centering{
\includegraphics[width=24em]{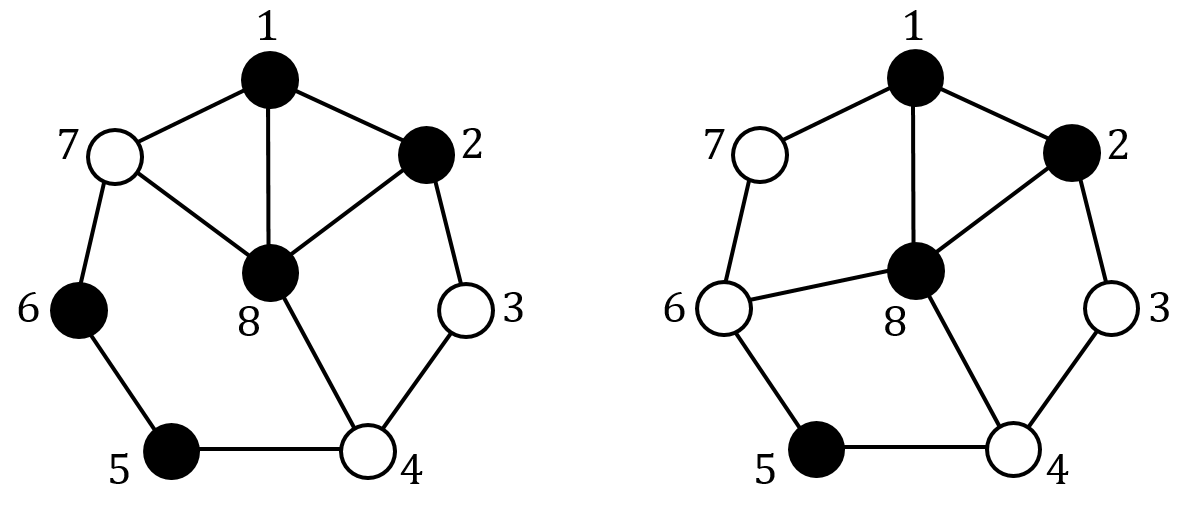}
}
\caption{Example graphs for lifting the hole inequality.}
\label{fig:holelift}
\end{figure}

Obviously, the structure of $N(v)$ in $G[H]$ becomes irrelevant at extreme values of $|H \cap N(v)|$. This result can be slightly strengthened by showing that the lengths of the paths among the neighbors of $v$ in $G[H]$ are redundant when $|H \cap N(v)|$ is restricted to 3 and 2 for $r=1$ and $r=2$, respectively. This leads to the following theorem concerning the lifting procedure of the hole inequality to higher-dimensional spaces.
\begin{theorem} \label{th:holelift}
Consider the hole inequality~\eqref{eq:hole} where $r \neq 0$. In every sequential lifting of~\eqref{eq:hole}, the coefficient of a variable $x_v,~v \in V \backslash H$, vanishes under the following conditions: {\em (i)} $|H \cap N(v)| \leq 3$ if $r=1$, {\em (ii)} $|H \cap N(v)| \leq 2$ if $r=2$. 
\end{theorem}
\begin{proof}
In an arbitrary sequential lifting of the hole inequality~\eqref{eq:hole}, even if some variables not-satisfying the theorem conditions have already been lifted into~\eqref{eq:hole} before $x_v$, existence of $2q + \left \lfloor \frac{2r}{3} \right \rfloor$ vertices in $H$ forming an IUC with $v$ is sufficient for the coefficient of $x_v$ to vanish. Thus, we just need to establish the existence of these vertices under the theorem conditions. Such structures for all possible distributions of neighbors of $v$ in $G[H]$ are presented below.

Consider a sequential partitioning of $H$ as described in the proof of Theorem~\ref{th:hole}. In a clockwise walk, let $v_i^p$ denote the $i$-th vertex in the $p$-th partition, where $i \in \{1,2,3\}$, $p \in \{1, \ldots, q+1\}$, and $p=q+1$ is the partition with less than 3 vertices. The cases where $|H \cap N(v)| \leq 1$ are trivial, so we start with $|H \cap N(v)| = 2$. Let $u$ and $w$ be the neighbors of $v$ in $H$, i.e., $H \cap N(v) = \{u,w\}$, and $d = d_{G[H]}(u,w)$ denote the distance (length of the shorter path) between $u$ and $w$ in $G[H]$. Without loss of generality, we assume $u = v_1^1$ and the shorter path between $u$ and $w$ in $G[H]$ is due to a clockwise walk from $u$ toward $w$. We treat $d \in \{1,2\}$ separately as those are the only cases where $u$ and $w$ are located in the same partition. For each case, we present an IUC of size $2q + \left \lfloor \frac{2r}{3} \right \rfloor + 1$ in $G[H \cup \{v\}]$. In the following presentation, each set of vertices is a clique, and $j$ denotes the partition that contains $w$. 
\begin{equation*}
\begin{aligned}
d=1,~r=1: \quad& \{v, u=v^1_1, w=v^1_2\};~\{v^p_1,v^p_2\}, \forall p \in \{2, \ldots, q\}\\
d=1,~r=2: \quad& \{v, u=v^1_1, w=v^1_2\};~\{v^p_1,v^p_2\}, \forall p \in \{2, \ldots, q\};~\{v^{q+1}_1\}\\
d=2,~r=1: \quad& \{v\};~\{v^1_2\};~\{v^p_1,v^p_2\}, \forall p \in \{2, \ldots, q\};~\{v^{q+1}_1\}\\
d=2,~r=2: \quad& \{v\};~\{v^1_2\};~\{v^p_1,v^p_2\}, \forall p \in \{2, \ldots, q\};~\{v^{q+1}_1,v^{q+1}_2\}\\
d \equiv 0~\text{(mod 3)},~r=1: \quad& \{v\};~\{v^p_2,v^p_3\}, \forall p \in \{1, \ldots, q\}\\
d \equiv 0~\text{(mod 3)},~r=2: \quad& \{v\};~\{v^p_2,v^p_3\}, \forall p \in \{1, \ldots, q\};~\{v^{q+1}_2\}\\
d \equiv 1~\text{(mod 3)},~r=1: \quad& \{v,w=v^j_2\};~\{v^p_2,v^p_3\}, \forall p \in \{1, \ldots, j-1\};\\
&\{v^p_1,v^p_2\}, \forall p \in \{j+1, \ldots, q\};~\{v^{q+1}_1\}\\
d \equiv 1~\text{(mod 3)},~r=2: \quad& \{v,w=v^j_2\};~\{v^p_2,v^p_3\}, \forall p \in \{1, \ldots, j-1\};\\
&\{v^p_1,v^p_2\}, \forall p \in \{j+1, \ldots, q\};~\{v^{q+1}_1,v^{q+1}_2\}\\
d \equiv 2~\text{(mod 3)},~r=1: \quad& \{v\};~\{v^p_2,v^p_3\}, \forall p \in \{1, \ldots, j-1\};~\{v^j_2\};\\
&\{v^p_1,v^p_2\}, \forall p \in \{j+1, \ldots, q\};~\{v^{q+1}_1\}\\
d \equiv 2~\text{(mod 3)},~r=2: \quad& \{v\};~\{v^p_2,v^p_3\}, \forall p \in \{1, \ldots, j-1\};~\{v^j_2\};\\
&\{v^p_1,v^p_2\}, \forall p \in \{j+1, \ldots, q\};~\{v^{q+1}_1,v^{q+1}_2\}\\
\end{aligned}
\end{equation*}
Next, we consider $|H \cap N(v)| = 3$ and $r=1$. Let $H \cap N(v) = \{u,w,z\}$. As before, suppose that $u = v_1^1$, and in a clockwise walk starting from $u$, first $w$ is reached in $d_1 = d_{G[H]}(u,w)$ steps, and then $z$ is reached in $d_2 = d_{G[H]}(w,z)$ steps from $w$. We also assume $d_1 \leq d_2$. In the following presentation, $j$ and $k$ denote the partitions containing $w$ and $z$, respectively. First, we consider the cases where $u$ and $w$ are located in the same partition, i.e., $d_1 \in \{1,2\}$. Note that $u$, $w$ and $z$ are placed in one partition only if $d_1=d_2=1$.  
\begin{equation*}
\begin{aligned}
d_1=1,~d_2=1: \quad& \{v, u=v^1_1, w=v^1_2\};~\{v^p_1,v^p_2\}, \forall p \in \{2, \ldots, q\}\\
d_1=1,~d_2 \equiv 0~\text{(mod 3)}: \quad& \{v, u=v^1_1, w=v^1_2\};~\{v^2_1\};\\
&\{v^p_3,v^{p+1}_1\}, \forall p \in \{2, \ldots, q-1\};~\{v^q_3\}\\
d_1=1,~d_2 \equiv 1~\text{(mod 3)}: \quad& \{v, u=v^1_1, w=v^1_2\};~\{v^p_1,v^p_2\}, \forall p \in \{2, \ldots, q\}\\
d_1=1,~d_2 \equiv 2~\text{(mod 3)}: \quad& \{v, u=v^1_1, w=v^1_2\};~\{v^p_2,v^p_3\}, \forall p \in \{2, \ldots, q\}\\
d_1=2,~d_2 \equiv 0~\text{(mod 3)}: \quad& \{v\};~\{v^1_2\};~\{v^p_1,v^p_2\}, \forall p \in \{2, \ldots, q\};~\{v^{q+1}_1\}\\
d_1=2,~d_2 \equiv 1~\text{(mod 3)}: \quad& \{v,z=v^k_1\};~\{v^1_2\};~\{v^p_1,v^p_2\}, \forall p \in \{2, \ldots, k-1\};\\
&\{v^p_3,v^{p+1}_1\}, \forall p \in \{k, \ldots, q\}\\
d_1=2,~d_2 \equiv 2~\text{(mod 3)}: \quad& \{v\};~\{v^1_2\};~\{v^2_1\};~\{v^p_3,v^{p+1}_1\}, \forall p \in \{2, \ldots, q\}\\
\end{aligned}
\end{equation*}
Finally, the following IUCs correspond to the cases where $u$, $w$ and $z$ are all located in separate partitions.
\begin{equation*}
\begin{aligned}
d_1 \equiv 0~\text{(mod 3)},~d_2 \equiv 0~\text{(mod 3)}: \quad& \{v\};~\{v^p_2,v^p_3\}, \forall p \in \{1, \ldots, q\}\\
d_1 \equiv 0~\text{(mod 3)},~d_2 \equiv 1~\text{(mod 3)}: \quad& \{v,z=v^k_2\};~\{v^p_2,v^p_3\}, \forall p \in \{1, \ldots, k-1\};\\
&\{v^p_1,v^p_2\}, \forall p \in \{k+1, \ldots, q\};~\{v^{q+1}_1\}
\end{aligned}
\end{equation*}
\begin{equation*}
\begin{aligned}
d_1 \equiv 0~\text{(mod 3)},~d_2 \equiv 2~\text{(mod 3)}: \quad& \{v\};~\{v^p_2,v^p_3\}, \forall p \in \{1, \ldots, k-1\};~\{v^k_2\};\\
&\{v^p_1,v^p_2\}, \forall p \in \{k+1, \ldots, q\};~\{v^{q+1}_1\}\\
d_1 \equiv 1~\text{(mod 3)},~d_2 \equiv 0~\text{(mod 3)}: \quad& \{v,w=v^j_2\};~\{v^p_2,v^p_3\}, \forall p \in \{1, \ldots, j-1\};\\
&\{v^{j+1}_1\};~\{v^p_3,v^{p+1}_1\}, \forall p \in \{j+1, \ldots, q\}\\
d_1 \equiv 1~\text{(mod 3)},~d_2 \equiv 1~\text{(mod 3)}: \quad& \{v,w=v^j_2\};~\{v^p_2,v^p_3\}, \forall p \in \{1, \ldots, j-1\};\\
&\{v^p_1,v^p_2\}, \forall p \in \{j+1, \ldots, q\};~\{v^{q+1}_1\}\\
d_1 \equiv 1~\text{(mod 3)},~d_2 \equiv 2~\text{(mod 3)}: \quad& \{v,u=v^1_1\};~\{v^p_3,v^{p+1}_1\}, \forall p \in \{1, \ldots, j-1\};\\
&\{v^p_3,v^{p+1}_1\}, \forall p \in \{j, \ldots, k-2\};~\{v^{k-1}_3\};\\
&\{v^p_2,v^p_3\}, \forall p \in \{k, \ldots, q\}
\end{aligned}
\end{equation*}
\begin{equation*}
\begin{aligned}
d_1 \equiv 2~\text{(mod 3)},~d_2 \equiv 0~\text{(mod 3)}: \quad& \{v\};~\{v^p_2,v^p_3\}, \forall p \in \{1, \ldots, j-1\};\\
&\{v^j_2\};~\{v^p_1,v^p_2\}, \forall p \in \{j+1, \ldots, q\};~\{v^{q+1}_1\}\\
d_1 \equiv 2~\text{(mod 3)},~d_2 \equiv 1~\text{(mod 3)}: \quad& \{v,z=v^k_1\};~\{v^p_2,v^p_3\}, \forall p \in \{1, \ldots, j-1\};\\
&\{v^j_2\};~\{v^p_1,v^p_2\}, \forall p \in \{j+1, \ldots, k-1\};\\
&\{v^p_3,v^{p+1}_1\}, \forall p \in \{k, \ldots, q\}\\
d_1 \equiv 2~\text{(mod 3)},~d_2 \equiv 2~\text{(mod 3)}: \quad& \{v\};~\{v^p_2,v^p_3\}, \forall p \in \{1, \ldots, j-1\};\\
&\{v^j_2\};~\{v^{j+1}_1\};~\{v^p_3,v^{p+1}_1\}, \forall p \in \{j+1, \ldots, q\}\\
\end{aligned}
\end{equation*}

The proof is complete.
\end{proof}

In Section~\ref{sec:fanwheel}, we will discuss implication of this theorem about the facial structure of the IUC polytope associated with a {\em defective wheel} graph.

We conclude this section by stating a similar result concerning the anti-hole inequality. By construction of a maximum IUC in an anti-cycle graph, presented it the proof of Theorem~\ref{th:antihole}, it is easy to verify that the coefficient of a variable $x_v, v \in V \backslash A$, vanishes in lifting the odd-anti-hole inequality~\eqref{eq:antihole} to higher-dimensional spaces if $|A \cap N(v)| \in \{0,1,|A|-1,|A|\}$. The following theorem mainly asserts that distribution of the neighbors of $v$ in $G[A]$ is not relevant in lifting $x_v$ into~\eqref{eq:antihole} when $|A \cap N(v)| \in \{2,|A|-2\}$, and the corresponding coefficient vanishes under this condition as well.
\begin{theorem} \label{th:antiholelift}
Consider the odd-anti-hole inequality~\eqref{eq:antihole}. In every sequential lifting of~\eqref{eq:antihole}, the coefficient of a variable $x_v, v \in V \backslash A$, vanishes if $|A \cap N(v)| \leq 2$ or $|A \cap N(v)| \geq |A|-2$. 
\end{theorem}
\begin{proof}
We show that if $|A \cap N(v)| \leq 2$, then $G[A]$ always contains a clique $C$ of cardinality $|C| = \frac{1}{2}(|A|-1)$ such that $C \cap N(v) = \emptyset$, and if $|A \cap N(v)| \geq |A|-2$, then $G[A \cup \{v\}]$ has a clique of size $\frac{1}{2}(|A|-1)+1$. The cases $|A \cap N(v)| \leq 1$ and $|A \cap N(v)| \geq |A|-1$ are trivial. Let $|A \cap N(v)| = 2$, and suppose $u$ and $w$ are the neighbors of $v$ in $A$. Consider a clockwise labeling of $A$ starting from $u$ as described in the proof of Theorem~\ref{th:antihole}. Observe that, if the label of $w$ is odd, then the set of vertices with even labels form a clique of size $\frac{1}{2}(|A|-1)$, neither of which is connected to $v$. On the other hand, if the label of $w$ is even, it will become odd in a counterclockwise labeling of $A$ starting from $u$, as $A$ itself is of odd cardinality. The clique of interest is then given by the even-labeled vertices in the counterclockwise labeling. The proof for the case $|A \cap N(v)| = |A|-2$ is identical except that $u$ and $w$ are defined to be the vertices in $A$ not adjacent to $v$. The proof is complete.
\end{proof}
\subsection{Star and double-star}
\label{sec:star}
\noindent
A {\em star} graph on $n$ vertices is a tree composed of a central vertex, called {\em hub}, connected to $n-1$ leaves. We use the notation $S =\{h\} \cup I$ to refer to the vertex set of a star graph. Here, $h$ is the hub vertex and $I$ denotes the set of leaves, remarking that they form an independent set in the graph. 
\begin{theorem} \label{th:star} \textbf{\em (Star Inequality)}
Let $S= \{h\} \cup I$ be a subset of vertices inducing a star graph in $G=(V,E)$. Then, the inequality
\begin{equation} \label{eq:star}
\sum_{i \in I} x_i + (|I| - 1) x_h \leq |I|
\end{equation}
is valid for $\mathscr{P}_{_{\!\!\mathcal{IUC}}}(G)$. Furthermore, {\em (i)} it induces a facet of $\mathscr{P}_{_{\!\!\mathcal{IUC}}}(G[S])$, and {\em (ii)} it is facet-defining for $\mathscr{P}_{_{\!\!\mathcal{IUC}}}(G)$ if and only if there does not exist a vertex $v \in V\backslash N[h]$ such that $S' = \{v\} \cup I$ induces another star graph in $G$.
\end{theorem}
\begin{proof}
Validity of~\eqref{eq:star} is evident. Let $\hat{x}$ be the incidence vector of an arbitrary IUC in $G$. If $\hat{x}_h = 0$, then~\eqref{eq:star} is trivially valid. Otherwise, at most one vertex from $I$ may belong to this IUC, since $\{h,i,j\} \in \Lambda(G), \forall i,j \in I$. 

(i) To prove that~\eqref{eq:star} is facet-defining for $\mathscr{P}_{_{\!\!\mathcal{IUC}}}(G[S])$, it is sufficient to point out that $x_I = \sum_{i \in I} e_i$ and $x_{\{h,i\}} = e_h + e_i, \forall i \in I$, are $|I|+1$ affinely independent points belonging to $\mathcal{F} = \{x \in \mathscr{P}_{_{\!\!\mathcal{IUC}}}(G[S]) \: | \: \sum_{i \in I} x_i + (|I| - 1) x_h = |I|\}$.

(ii) To see the necessity of this condition, note that if there exist two star subgraphs in $G$ induced by $S= \{h\} \cup I$ and $S' = \{v\} \cup I$, for some $v \in V \backslash N[h]$, then the star inequality corresponding to $S$ is dominated by $\sum_{i \in I} x_i + (|I|-1) x_h + x_v \leq |I|$, which is obtained from lifting $x_v$ into~\eqref{eq:star}. 

To prove sufficiency, assume that there does not exist a vertex $v \in V\backslash N[h]$ such that $S' = \{v\} \cup I$ induces a star graph in $G$.
Let $\sum_{v \in V \backslash S} a_v x_v + \sum_{i \in I} x_i + (|I| - 1) x_h \leq |I|$ be the inequality obtained from an arbitrary sequential lifting of~\eqref{eq:star}. We need to show that $a_v = 0, \forall v \in V \backslash S$. In this regard, consider the following cases: 
\begin{itemize}
\item $v \in N(h)$\\
If $v$ is also adjacent to a vertex $i \in I$, then $\{h,i,v\}$ is a clique. Otherwise, $\{v\} \cup I$ is an independent set in $G$.    

\item $v \notin N(h)$\\
Since $\{v\} \cup I$ does not induce a star subgraph, there exists a vertex $i \in I$ not adjacent to $v$. Thus, $G[\{h,i,v\}]$ is the union of an isolated vertex $v$ and an edge $\{h,i\}$.
\end{itemize}
Therefore, there always exists an IUC in $G[S \cup \{v\}]$ containing $v$ whose incidence vector satisfies $\sum_{i \in I} x_i + (|I| - 1) x_h = |I|$. This results implies that $a_v = 0, \forall v \in V \backslash S$, hence~\eqref{eq:star} is facet-defining for $\mathscr{P}_{_{\!\!\mathcal{IUC}}}(G)$ under the theorem's condition.
\end{proof}

It is interesting to note that no further condition, such as maximality of $I$ in the neighborhood of $h$, is required for~\eqref{eq:star} to be facet-defining for $\mathscr{P}_{_{\!\!\mathcal{IUC}}}(G)$. In fact, in absence of a vertex $v \in V \backslash N[h]$ described in the theorem statement, every $I' \subseteq I, |I'| \geq 2$, corresponds to a distinct facet of $\mathscr{P}_{_{\!\!\mathcal{IUC}}}(G)$ defined by~\eqref{eq:star}. In particular, every minimal (inclusion-wise) of such subsets corresponds to a facet-defining OT inequality. Recall that an OT inequality corresponding to $\{h,i,j\}$, for some $\{i,j\} \subseteq I$, is facet-defining for $\mathscr{P}_{_{\!\!\mathcal{IUC}}}(G)$ if and only if $\{h,i,j,v\}, \forall v \in V \backslash \{h,i,j\}$, is not a 4-hole, or equivalently, $\{v,i,j\}$ does not induce a star subgraph.  

The proof of Theorem~\ref{th:star} also declares that if two star subgraphs with the same set of leaves exist in $G$, then the family of star inequalities~\eqref{eq:star} corresponding to each one of them can be lifted to a higher-dimensional space to generate facet-inducing inequalities for the corresponding {\em double-star} (complete bipartite $K_{|I|,2}$) subgraph. The following theorem shows that these double-star inequalities are actually facet-inducing for $\mathscr{P}_{_{\!\!\mathcal{IUC}}}(G)$, which also generalizes the facial property of the 4-hole inequality.
\begin{theorem} \label{th:Dstar} \textbf{\em (Double-star Inequality)}
Let $S_h= \{h\} \cup I$ and $S_u= \{u\} \cup I$ be two sets of vertices inducing star subgraphs in $G$, where $h$ and $u \in V \backslash N[h]$ denote two (non-adjacent) hub vertices, and $I$ is the joint set of leaves. Then, the inequalities
\begin{equation} \label{eq:doublestar}
\sum_{i \in I} x_i + (|I|-1) x_h + x_u \leq |I| 
\text{\quad and \quad}
\sum_{i \in I} x_i + x_h + (|I|-1) x_u \leq |I|
\end{equation}
induce facets of $\mathscr{P}_{_{\!\!\mathcal{IUC}}}(G)$.
\end{theorem}
\begin{proof}
We present the proof for $\sum_{i \in I} x_i + (|I|-1) x_h + x_u \leq |I|$. The other follows from the symmetry of this structure. To this end, we just need to show that, for every $v \in V \backslash (\{h,u\} \cup I )$, there always exists an IUC containing $v$ in the subgraph $G[\{h,u,v\} \cup I]$ whose incidence vector satisfies $\sum_{i \in I} x_i + (|I|-1) x_h + x_u = |I|$. This result will be sufficient to indicate that the coefficient of a variable $x_v, \forall v \in V \backslash (\{h,u\} \cup I )$ vanishes in every sequential lifting of $\sum_{i \in I} x_i + (|I|-1) x_h + x_u \leq |I|$, proving that this inequality is facet-defining for $\mathscr{P}_{_{\!\!\mathcal{IUC}}}(G)$. 
The existence of such IUCs when $v \in N(h)$ is established by the proof of Theorem~\ref{th:star}. On the other hand, if $v \notin N(h)$, then $\{h,u,v\}$ is an IUC, that is either an independent set or the union of an isolated vertex with two adjacent vertices. This completes the proof.
\end{proof}

It is easy to see that if $G=(S,E),~S = \{h\} \cup I$, is a star graph, then a subset of vertices in $G$ is an IUC if and only if it is a {\em co-1-defective clique}. An $s$-defective clique is a subset of vertices $C$ such that the corresponding induced subgraph contains at least $\binom{|C|}{2} - s$ edges~\cite{pattilloA2013}. Correspondingly, a co-$s$-defective clique is a subset of vertices whose induced subgraph has at most $s$ edges. Sherali and Smith~\cite{sheraliSmith2006} have shown that the family of inequalities defined by~\eqref{eq:star} for all $I' \subseteq I, |I'| \geq 2$, along with the variable bounds, is sufficient to describe the co-1-defective clique polytope associated with a star graph. Given the equivalence of these two structures in a star graph, this result also holds for the IUC polytope associated with $G$, which is presented through the following theorem.
\begin{theorem} \label{th:starfull}
The family of inequalities~\eqref{eq:star} together with the variable bounds provide a complete description for the IUC polytope associated with a star graph $G=(S,E)$.
\end{theorem}
\begin{proof}
See Proposition 3 in~\cite{sheraliSmith2006} and its proof.
\end{proof}

In~\cite{sheraliSmith2006}, the authors use the term {\em generalized vertex packing} to refer to co-$s$-defective clique. We have used the latter term to avoid confusion with the concept of {\em generalized independent set}, which has also been the subject of some recent studies~\cite{colombiA2017,hosseinianA2019GISP}.

We extend this result to the IUC polytope associated with a (general) tree. Every tree can be seen as a collection of connected stars, whose hub vertices are the internal (branch) nodes of the tree. The following theorem declares that the family of star inequalities~\eqref{eq:star} generated by the star subgraphs of a tree, together with the variable bounds, is sufficient to describe the IUC polytope associated with this graph. 

\begin{theorem} \label{th:tree}
Let $G=(V,E)$ be a tree, and $V_b \subset V$ denote the set of its internal (branch) vertices. Then,
\begin{equation} \label{eq:treefull}
\mathscr{P}_{_{\!\!\mathcal{IUC}}}(G) = \{x \in [0,1]^n \: | \: \sum_{i \in I} x_i + (|I| - 1) x_v \leq |I|,~\forall v \in V_b,~\forall I \subseteq N(v), |I| \geq 2 \}.
\end{equation}
\end{theorem}
\begin{proof}
We show that every extreme point of the polytope defined by the inequalities in~\eqref{eq:treefull} is integral. The proof is by induction. The base case is a star graph, for which the result holds by Theorem~\ref{th:starfull}. Patently, every tree $G=(V,E)$ is constructed by appending a leaf vertex $j \in V$ to the subtree $G[V \backslash \{j\}]$. So, suppose that the result holds for every induced subtree of $G[V \backslash \{j\}]$; we show that it also holds for $G$. Let $h$ denote the parent of $j$, i.e., the internal vertex adjacent to $j$, in $G$, and without loss of generality, assume that $G$ is rooted at $h$. Furthermore, let $N(h) \backslash \{j\} = \{r_k,~\forall k \in \{1,\ldots,|N(h)|-1\}\}$, and $G^{k},~\forall k \in \{1,\ldots,|N(h)|-1\}\}$, denote the subtree of $G$ rooted at $r_k$ together with $h$ as a leaf vertex. We present the proof for the case that $h$ is an internal vertex of $G[V \backslash \{j\}]$; the case that $h$ is a leaf vertex in $G[V \backslash \{j\}]$ becomes trivial after presenting this (more general) result, as it corresponds to $k \in \{1\}$. 

Let $\hat{x}$ be an arbitrary extreme point of the polytope defined by the inequalities in~\eqref{eq:treefull}, denoted by $\mathscr{P}$. Evidently, if $\hat{x}_j = 0$ or $\hat{x}_h = 0$, the result trivially holds, by the induction hypothesis. Besides, $\hat{x}_j = \hat{x}_h = 1$ implies $\hat{x}_{r_k} = 0,~\forall k \in \{1,\ldots,|N(h)|-1\}\}$, which brings about the same result. We show that fractionality of $\hat{x}_j$ or $\hat{x}_h$ contradicts the extremity of $\hat{x}$ for each one of the remaining three possibilities:
\begin{itemize}
\item[(i)] $\hat{x}_h = 1$~~and~~$\hat{x}_j \in (0,1)~~\Rightarrow~~\hat{x}_{r_k} \neq 1,~\forall k \in \{1,\ldots,|N(h)|-1\}\}$\\
Consider the facet $\mathscr{P}_{(x_h = 1)} = \{x \in \mathscr{P} \: | \: x_h = 1\}$. Observe that, ${x}_h = 1$ reduces the entire set of star inequalities~\eqref{eq:star} associated with the star subgraph $G[\{h\} \cup N(h)]$ to 
$$
x_j + \sum_{k=1}^{|N(h)|-1} x_{r_k} \leq 1, 
$$
in the description of $\mathscr{P}_{(x_h = 1)}$, and every other inequality in the description of this facet uniquely appears in the description of the polytope $\mathscr{P}_{(x_h = 1)}^{k} = \{x \in \mathscr{P}_{_{\!\!\mathcal{IUC}}}(G^{k}) \: | \: x_h = 1\}$, for some $k \in \{1,\ldots,|N(h)|-1\}\}$, by the induction hypothesis. Note that, if $\hat{x}_{r_k} = 0$, for some $k \in \{1,\ldots,|N(h)|-1\}\}$, we may exclude the subtree rooted at $r_k$ from the graph; so, without loss of generality, we assume $\hat{x}_{r_k} \in (0,1),~\forall k \in \{1,\ldots,|N(h)|-1\}\}$. Let $V^1 = V (G^{1})$ denote the vertex set of $G^{1}$, and consider the point $\tilde{x} \in \mathbb{R}^{|V^1|}$ given by $\tilde{x}_i = \hat{x}_i,~\forall i \in V^1$. By the induction hypothesis, $\tilde{x}$ is not an extreme point of $\mathscr{P}_{(x_h = 1)}^{1} = \{x \in \mathscr{P}_{_{\!\!\mathcal{IUC}}}(G^{1}) \: | \: x_h = 1\}$. Hence, there always exist two distinct points $\tilde{x}^+ \in \mathscr{P}_{(x_h = 1)}^{1}$ and $\tilde{x}^- \in \mathscr{P}_{(x_h = 1)}^{1}$ such that
$\tilde{x}^+_{r_1} = \tilde{x}_{r_1} + \epsilon$ and $\tilde{x}^-_{r_1} = \tilde{x}_{r_1} - \epsilon$, for some $\epsilon \geq 0$, and $\tilde{x} = \frac{1}{2} (\tilde{x}^+ + \tilde{x}^-)$. Notice that $\tilde{x}^+$ and $\tilde{x}^-$ are always two distinct points, and $\epsilon = 0$ corresponds to the event that they differ in a coordinate other than $x_{r_1}$. This immediately implies that $\hat{x}$ can always be written as a convex combination of two distinct points $\hat{x}^+ \in \mathscr{P}_{(x_h = 1)}$ and $\hat{x}^- \in \mathscr{P}_{(x_h = 1)}$ given by
$$
\begin{aligned}
&\hat{x}^+_{j} = \hat{x}_{j} - \epsilon;~~\hat{x}^+_{h} = 1;~~\hat{x}^+_{r_1} = \hat{x}_{r_1} + \epsilon,~~\hat{x}^+_{i} = \tilde{x}^+_{i},\forall i \in V^1 \backslash \{h,r_1\};~~\hat{x}^+_{i} = \hat{x}_{i},\forall i \in V \backslash (\{j\} \cup V^1),\\
&\hat{x}^-_{j} = \hat{x}_{j} + \epsilon;~~\hat{x}^-_{h} = 1;~~\hat{x}^-_{r_1} = \hat{x}_{r_1} - \epsilon,~~\hat{x}^-_{i} = \tilde{x}^-_{i},\forall i \in V^1 \backslash \{h,r_1\};~~\hat{x}^-_{i} = \hat{x}_{i},\forall i \in V \backslash (\{j\} \cup V^1),
\end{aligned}
$$
which contradicts the extremity of $\hat{x}$.
\item[(ii)] $\hat{x}_h \in (0,1)$~~and~~$\hat{x}_j = 1~~\Rightarrow~~\hat{x}_{r_k} \neq 1,~\forall k \in \{1,\ldots,|N(h)|-1\}\}$\\
As before, we assume $\hat{x}_{r_k} \in (0,1),~\forall k \in \{1,\ldots,|N(h)|-1\}\}$. In this case, the set of star inequalities~\eqref{eq:star} associated with the star subgraph $G[\{h\} \cup N(h)]$ reduces to 
\begin{equation} \label{eq:hstar}
\sum_{i \in I} x_i + |I| x_h \leq |I|,~\forall I \subseteq N(h) \backslash \{j\}, |I| \geq 1.  
\end{equation}
Notice that every inequality not involving $x_j$ in this set is redundant in the description of $\mathscr{P}_{(x_j = 1)} = \{x \in \mathscr{P} \: | \: x_j = 1\}$, and~\eqref{eq:hstar} is due to substituting $x_j = 1$ in the inequalities involving this variable. Besides, every other inequality in the description of $\mathscr{P}_{(x_j = 1)}$ uniquely appears in the description of $\mathscr{P}_{_{\!\!\mathcal{IUC}}}(G^{k})$, for some $k \in \{1,\ldots,|N(h)|-1\}\}$, by the induction hypothesis. Observe that every two points $\hat{x}^+ \in \mathbb{R}^n$ and $\hat{x}^- \in \mathbb{R}^n$ that satisfy
\begin{equation} \label{eq:hstar2}
\begin{aligned}
&\hat{x}^{+}_h = \hat{x}_h + \epsilon;~~\hat{x}^{+}_{r_k} = \hat{x}_{r_k} - \epsilon,\forall k \in \{1,\ldots,|N(h)|-1\},\\
&\hat{x}^{-}_h = \hat{x}_h - \epsilon;~~\hat{x}^{-}_{r_k} = \hat{x}_{r_k} + \epsilon,\forall k \in \{1,\ldots,|N(h)|-1\},
\end{aligned}
\end{equation}
for some $\epsilon \geq 0$, will also satisfy all inequalities in~\eqref{eq:hstar}. We aim to show that, under the induction assumptions, $\hat{x}$ can always be written as $\hat{x} = \frac{1}{2} (\hat{x}^+ + \hat{x}^-)$ for two distinct points $\hat{x}^+ \in \mathbb{R}^n$ and $\hat{x}^- \in \mathbb{R}^n$ that, in addition to~\eqref{eq:hstar2} and $\hat{x}^+_{j} = \hat{x}^-_{j} = \hat{x}_{j} = 1$, satisfy all defining inequalities of $\mathscr{P}_{_{\!\!\mathcal{IUC}}}(G^{k}),~\forall k \in \{1,\ldots,|N(h)|-1\}\}$. Clearly, we just need to present the proof for $\mathscr{P}_{_{\!\!\mathcal{IUC}}}(G^{1})$, as $x_h$ is the only variable that the defining inequalities of $\mathscr{P}_{_{\!\!\mathcal{IUC}}}(G^{k}),~\forall k \in \{1,\ldots,|N(h)|-1\}\}$, share. To this end, it is sufficient to show that the point $\tilde{x} \in \mathbb{R}^{|V^1|}$ given by $\tilde{x}_i = \hat{x}_i,~\forall i \in V^1$, is not an extreme point of the polytope
$$
\widetilde{\mathscr{P}} = \mathscr{P}_{_{\!\!\mathcal{IUC}}}(G^{1})~\cap~\{x \in \mathbb{R}^{|V^1|} \: | \: x_{r_1} + x_h = \tilde{x}_{r_1} + \tilde{x}_h \}.
$$
By the induction hypothesis, $\tilde{x}$ is not an extreme point of $\mathscr{P}_{_{\!\!\mathcal{IUC}}}(G^{1})$. Therefore, if it is an extreme point of $\widetilde{\mathscr{P}}$, it must lie on the interior of an edge of $\mathscr{P}_{_{\!\!\mathcal{IUC}}}(G^{1})$, and can be written as
$
\tilde{x} = \lambda x^{U_1} + (1-\lambda) x^{U_2},~ \lambda \in (0,1),
$
where $x^{U_1}$ and $x^{U_2}$ are the incidence vectors of two distinct IUCs in $G^1$. Since $\tilde{x}_{r_1}$ and $\tilde{x}_{h}$ are not integral, $x^{U_1}$ and $x^{U_2}$ must satisfy one of the following conditions:
\begin{itemize}
\item[(a)] $x^{U_1}_{r_1} = x^{U_2}_h = 1$~~and~~$x^{U_2}_{r_1} = x^{U_1}_h = 0$\\
Thus, $\tilde{x}_{r_1} = \lambda$, $\tilde{x}_h = 1- \lambda$, and $\tilde{x}_{r_1} + \tilde{x}_h = 1$. This implies that every star inequality~\eqref{eq:star} involving $x_h$, associated with the star subgraph $G^1[\{r_1\} \cup N(r_1)]$, is redundant in the description of $\widetilde{\mathscr{P}}$, as it can be written as a linear combination of $x_{r_1} + x_h \leq 1$ and another star inequality, associated with the subgraph $G^1[S],~S = \{r_1\} \cup (N(r_1)  \backslash \{h\})$, or a variable bound. Consider the point $\bar{x} \in \mathbb{R}^{|V^1|-1}$ given by $\bar{x}_i = \tilde{x}_i,~\forall i \in V^1 \backslash \{h\}$. By the induction hypothesis, $\bar{x}$ can always be written as a convex combination of two distinct points $\bar{x}^+ \in \mathscr{P}_{_{\!\!\mathcal{IUC}}}(G^{1}[V^1 \backslash \{h\}])$ and $\bar{x}^- \in \mathscr{P}_{_{\!\!\mathcal{IUC}}}(G^{1}[V^1 \backslash \{h\}])$ such that $\bar{x}^+_{r_1} = \bar{x}_{r_1} + \epsilon$ and $\bar{x}^-_{r_1} = \bar{x}_{r_1} - \epsilon$, for some $\epsilon \geq 0$. Since ${x}_{r_1} + {x}_h = 1$ is the only essential inequality involving $x_h$ in the description $\widetilde{\mathscr{P}}$, this result immediately implies that $\bar{x}$ can always be written as a convex combination of two distinct points $\tilde{x}^+ \in \widetilde{\mathscr{P}}$ and $\tilde{x}^- \in \widetilde{\mathscr{P}}$, given by
\begin{equation} \label{eq:point}
\begin{aligned}
&\tilde{x}^+_h = \tilde{x}_h - \epsilon;~~\tilde{x}^+_{r_1} = \tilde{x}_{r_1} + \epsilon;~~\tilde{x}^+_i = \bar{x}^+_i,~\forall i \in V^1 \backslash \{h,r_1\},\\
&\tilde{x}^-_h = \tilde{x}_h + \epsilon;~~\tilde{x}^-_{r_1} = \tilde{x}_{r_1} - \epsilon;~~\tilde{x}^-_i = \bar{x}^-_i,~\forall i \in V^1 \backslash \{h,r_1\},
\end{aligned}
\end{equation}
which contradicts the extremity of $\tilde{x}$. 
\item[(b)] $x^{U_1}_{r_1} = x^{U_1}_h = 1$~~and~~$x^{U_2}_{r_1} = x^{U_2}_h = 0~~\Rightarrow~~x^{U_1}_i = 0,~\forall i \in N(r_1) \backslash \{h\}$\\
Thus, $\tilde{x}_{r_1} = \tilde{x}_h = \lambda$ and $\tilde{x}_i \in \{0,1-\lambda\},~\forall i \in N(r_1) \backslash \{h\}$. This implies that no star inequality~\eqref{eq:star} involving $x_h$, associated with the star subgraph $G^1[\{r_1\} \cup N(r_1)]$, can be binding at $\tilde{x}$, because otherwise the equality 
$$
\sum_{i \in I} x_i + x_h + |I| x_{r_1} = |I|+1,~ I \subseteq N(r_1) \backslash \{h\},~|I| \geq 1,
$$
leads to $\lambda = 1$. Similar to (a), this result further implies the existence of two distinct points $\tilde{x}^+ \in \widetilde{\mathscr{P}}$ and $\tilde{x}^- \in \widetilde{\mathscr{P}}$, defined in~\eqref{eq:point} for a small enough $\epsilon$, which contradicts the extremity of $\tilde{x}$.
\end{itemize}
\item[(iii)] $\hat{x}_h \in (0,1)$~~and~~$\hat{x}_j \in (0,1)$\\
Consider the following possibilities:
\begin{itemize}
\item[(a)] $\hat{x}_h + \hat{x}_j > 1~~\Rightarrow~~\hat{x}_{r_k} \in [0,1),~\forall k \in \{1,\ldots,|N(h)|-1\}$\\
In this case, no star inequality~\eqref{eq:star} associated with the star subgraph $G[S],~S = \{h\} \cup (N(h) \backslash \{j\})$, can be tight at $\hat{x}$, because otherwise $\hat{x}$ violates a star inequality involving $x_j$. Therefore, any star inequality~\eqref{eq:star} associated with the star subgraph $G[\{h\} \cup N(h)]$ binding at $\hat{x}$ is of the form 
$$
\sum_{i \in I} x_i + x_j + |I| x_h \leq |I|+1,~\forall I \subseteq N(h) \backslash \{j\}, |I| \geq 1,  
$$
and every other inequality in the description of $\mathscr{P}$ uniquely appears in the description of $\mathscr{P}_{_{\!\!\mathcal{IUC}}}(G^{k})$, for some $k \in \{1,\ldots,|N(h)|-1\}\}$, by the induction hypothesis. The rest of the proof is identical to (ii), except that here we have $\hat{x}^+_{j} = \hat{x}^-_{j} = \hat{x}_{j} \neq 1$.
\item[(b)] $\hat{x}_h + \hat{x}_j = 1$\\
If $\hat{x}$ is an extreme point of $\mathscr{P}$, it must also be an extreme point of the polytope
$$
\widetilde{\mathscr{P}} = \mathscr{P}~\cap~\{x \in \mathbb{R}^n \: | \: x_{h} + x_j = 1 \}.
$$
Then, an argument similar to (ii).(a) will lead to a contradiction with the extremity of $\hat{x}$.
\item[(c)] $\hat{x}_h + \hat{x}_j < 1$\\
This implies that no star inequality~\eqref{eq:star} involving $x_j$, associated with the star subgraph $G[\{h\} \cup N(h)]$,
can be binding at $\hat{x}$, because otherwise $\hat{x}$ violates a star inequality~\eqref{eq:star} associated with the star subgraph $G[S],~S=\{h\} \cup (N(h) \backslash \{j\})$. Since the other inequalities in the description of $\mathscr{P}$ are the inequalities defining $\mathscr{P}_{_{\!\!\mathcal{IUC}}}(G[V \backslash \{j\}])$, by the induction hypothesis, this is a blatant contradiction with the extremity of $\hat{x}$.
\end{itemize}
\end{itemize}
The proof is complete.
\end{proof}

Given a graph $G$, a subset of vertices is called a {\em $k$-dependent set} if the degree of every vertex in the corresponding induced subgraph is at most $k$. It is evident that IUC and $1$-dependent set are equivalent in trees. Hence, the foregoing result also holds for the $1$-dependent set polytope associated with a tree, which we present through the following corollary. We should also point out that the {\sc Maximum $k$-dependent Set} problem, which is to find a maximum-cardinality $k$-dependent set in $G$, is linear-time solvable on trees, for every $k \geq 1$~\cite{DessmarkA93}.

\begin{corollary}
Let $G$ be a tree. Then, the family of star inequalities~\eqref{eq:star} for all star subgraphs of $G$, together with the variable bounds, is sufficient to describe the 1-dependent set polytope associated with $G$.
\end{corollary}

We also show that the family of inequalities given by~\eqref{eq:doublestar} plays the same role in description of the IUC polytope associated with a complete bipartite graph $G = K_{|I|,|J|}$. The class of complete bipartite graphs is of a special interest because it has been shown by Pyatkin et al.~\cite{pyatkinA2018} that a complete bipartite graph $K_{\lfloor \frac{n}{2} \rfloor, \lceil \frac{n}{2} \rceil }$ on $n$ vertices accommodates the maximum possible number of open triangles among all the graphs with the same number of vertices.
\begin{theorem} \label{th:CB}
Let $G=(V,E)$ be a complete bipartite graph, where the bipartition of vertices is given by $V = I \cup J$. Then, a complete description of the IUC polytope associated with $G$ is given by the family of inequalities~\eqref{eq:doublestar} for all double-star subgraphs of $G$ together with the variable bounds.
\end{theorem}
\begin{proof}
We assume $|I|, |J| \geq 2$ because otherwise $G$ is a star graph. Let $\mathscr{P}$ be the polytope defined by the family of inequalities~\eqref{eq:doublestar} for all double-star subgraphs of $G$ together with the variable bounds:
\begin{subequations} \label{eq:CBfull}
\begin{align} 
\mathscr{P}= \{ x \in \mathbb{R}^n \: |~& 0 \leq x_v \leq 1, \; \forall v \in V; \label{eq:CBfull_e} \\
\nonumber&\forall \{j',j''\} \subseteq J,~\forall I' \subseteq I, |I'| \geq 2:\\
&\hspace{6mm}\sum_{i \in I'} x_i + (|I'|-1) x_{j'} + x_{j''} \leq |I'|, \label{eq:CBfull_a}\\
&\hspace{6mm}\sum_{i \in I'} x_i + x_{j'} + (|I'|-1) x_{j''} \leq |I'|;\label{eq:CBfull_b}\\
\nonumber&\forall \{i',i''\} \subseteq I,~\forall J' \subseteq J, |J'| \geq 2: \\
&\hspace{6mm}\sum_{j \in J'} x_j + (|J'|-1) x_{i'} + x_{i''} \leq |J'|, \label{eq:CBfull_c}\\
&\hspace{6mm}\sum_{j \in J'} x_j + x_{i'} + (|J'|-1) x_{i''} \leq |J'|\}.\label{eq:CBfull_d}
\end{align}
\end{subequations}
We aim to show that $\mathscr{P}_{_{\!\!\mathcal{IUC}}}(G) = \mathscr{P}$. Validity of the set of inequalities~\eqref{eq:CBfull} for the IUC polytope associated with $G$ indicates that $\mathscr{P}_{_{\!\!\mathcal{IUC}}}(G) \subseteq \mathscr{P}$. In order to prove that the reverse also holds, we show that every extreme point of $\mathscr{P}$ is integral, hence it is the incidence vector of an IUC in $G$, which implies $\mathscr{P} \subseteq \mathscr{P}_{_{\!\!\mathcal{IUC}}}(G)$.
The proof is by contradiction. Let $\hat{x}$ be a fractional extreme point of $\mathscr{P}$, and define $V_0 = \{v \in V \: | \: \hat{x}_v = 0\}$ and $V_1 = \{v \in V \: | \: \hat{x}_v = 1\}$. First, observe that every inequality in~\eqref{eq:CBfull} involving $x_v$, for some $v \in V_0$, is redundant in description of the face $\mathscr{P}_{(x_{V_0} = 0)} = \{x \in \mathscr{P} \: | \: x_v = 0, \forall v \in V_0\}$. To see this, consider a fixed vertex $\tilde{j} \in J \cap V_0$. If $|J| = 2$, i.e., $J = \{\tilde{j},j'\}$, then~\eqref{eq:CBfull} reduces to 
$$
\begin{aligned}
\mathscr{P}=\{ x \in \mathbb{R}^n \: |~&0 \leq x_{j'} \leq 1;~~0 \leq x_i \leq 1,  \forall i \in I;\\
& \sum_{i \in I'} x_i + (|I'|-1) x_{j'} \leq |I'|, \forall I' \subseteq I, |I'| \geq 2\},
\end{aligned}
$$
which is precisely the description of the IUC polytope associated with the star subgraph $G[\{j'\} \cup I]$. If $|J| \geq 3$, then each inequality of type~\eqref{eq:CBfull_a}-\eqref{eq:CBfull_b} corresponding to $\{\tilde{j},j'\}$ is dominated by an inequality of the same type corresponding to $\{j',j''\},\forall j'' \in J \backslash \{\tilde{j},j'\}$. Besides, each inequality of type~\eqref{eq:CBfull_c}-\eqref{eq:CBfull_d} involving $x_{\tilde{j}}$ can be written as a linear combination of $x_{i'} \leq 1$ (or $x_{i''} \leq 1$) and another inequality of the same type corresponding to $J' \backslash \{\tilde{j}\}$. Therefore,
\begin{subequations} \label{eq:V0}
\begin{align}
\mathscr{P}_{(x_{V_0} = 0)} = \{ x \in \mathbb{R}^n \: |~&  x_v = 0,~\forall v \in V_0; \\
&0 \leq x_v \leq 1,~\forall v \in V \backslash V_0; \label{eq:V0_e}\\
\nonumber &\forall \{j',j''\} \subseteq J \backslash V_0,~\forall I' \subseteq I \backslash V_0, |I'| \geq 2: \\
&\;\;\;\; \;\; \sum_{i \in I'} x_i + (|I'|-1) x_{j'} + x_{j''} \leq |I'|, \label{eq:V0_a}\\
&\;\;\;\; \;\; \sum_{i \in I'} x_i + x_{j'} + (|I'|-1) x_{j''} \leq |I'|;\label{eq:V0_b}\\
\nonumber &\forall \{i',i''\} \subseteq I \backslash V_0,~\forall J' \subseteq J \backslash V_0, |J'| \geq 2: \\
&\;\;\;\; \;\; \sum_{j \in J'} x_j + (|J'|-1) x_{i'} + x_{i''} \leq |J'|, \label{eq:V0_c}\\
&\;\;\;\; \;\; \sum_{j \in J'} x_j + x_{i'} + (|J'|-1) x_{i''} \leq |J'|\}.\label{eq:V0_d}
\end{align}
\end{subequations}
Now, consider the following three possibilities for the cardinality of $V_1$: $|V_1| = 0,$ $|V_1| = 1$, and $|V_1| \ge2$.
\begin{itemize}
\item \textsc{case 1}: $|V_1| = 0$ ($\hat{x}_v$ is fractional for every $v \in V \backslash V_0$.)\\
Then, there exists a sufficiently small $\epsilon > 0$ for which $\hat{x}^+$ and $\hat{x}^-$ with
$$
\begin{aligned}
&\hat{x}^+_i = \hat{x}_i + \epsilon, \forall i \in I \backslash V_0;~\hat{x}^+_j = \hat{x}_j - \epsilon, \forall j \in J \backslash V_0;~\hat{x}^+_v = \hat{x}_v, \forall v \in V_0  \\
&\hat{x}^-_i = \hat{x}_i - \epsilon, \forall i \in I \backslash V_0;~\hat{x}^-_j = \hat{x}_j + \epsilon, \forall j \in J \backslash V_0;~\hat{x}^-_v = \hat{x}_v, \forall v \in V_0
\end{aligned}
$$
belong to $\mathscr{P}_{(x_{V_0} = 0)} \subseteq \mathscr{P}$. Since $\hat{x} = \frac{1}{2}(\hat{x}^+ + \hat{x}^-)$, this contradicts extremity of $\hat{x}$.  
\item \textsc{case 2}: $|V_1|=1$\\
Without loss of generality, assume $V_1 = \{\tilde{j}\}$, for some $\tilde{j}\in J$. Note that, in description of the face $\mathscr{P}_{(x_{V_0} = 0, x_{\tilde{j}}=1)}$, every inequality in~\eqref{eq:V0_b} corresponding to $\{\tilde{j},j'\}, j' \in J \backslash (V_0 \cup \{\tilde{j}\})$, is dominated by the corresponding inequality of type~\eqref{eq:V0_a}, as $x_{j'} \leq 1$. Furthermore, each inequality of type~\eqref{eq:V0_a} involving $x_{\tilde{j}}$ and corresponding to a proper subset of $I$ is in turn dominated by the one corresponding to $I$ itself, i.e., 
\begin{equation} \label{eq:Iitself}
\sum_{i \in I} x_i + x_{j'} \leq 1,~\forall j' \in J \backslash (V_0 \cup \{\tilde{j}\}).
\end{equation}
Now, it is clear that every inequality of type~\eqref{eq:V0_a}-\eqref{eq:V0_b} corresponding to $\{j',j''\} \subseteq J \backslash (V_0 \cup \{\tilde{j}\})$ and $I$ can be written as a linear combination of an inequality in~\eqref{eq:Iitself}, $x_{j'} \leq 1$, and $x_{j''} \leq 1$, and the ones corresponding to proper subsets of $I$ are dominated by such linear combinations. Hence,~\eqref{eq:V0_a}-\eqref{eq:V0_b} reduces to~\eqref{eq:Iitself} in description of $\mathscr{P}_{(x_{V_0} = 0, x_{\tilde{j}}=1)} = \{x \in \mathscr{P} \: | \: x_{\tilde{j}}=1;~x_v = 0, \forall v \in V_0\}$.
Besides, in the description of this face, every inequality of type~\eqref{eq:V0_c}-\eqref{eq:V0_d} corresponding to a set $J'=\tilde{J},~\tilde{j} \notin \tilde{J}$, is dominated by the inequality corresponding to $J' = \tilde{J} \cup \{\tilde{j}\}$, thus~\eqref{eq:V0_c}-\eqref{eq:V0_d} can be replaced by
\begin{subequations} \label{eq:V0_cdreduce}
\begin{align}
\nonumber\forall \{i',i''\} \subseteq I \backslash V_0,&~\forall J' \subseteq J \backslash (V_0 \cup \{\tilde{j}\}),  |J'| \geq 1: \\
&\;\;\;\;\sum_{j \in J'} x_j + |J'| x_{i'} + x_{i''} \leq |J'|, \hspace{10mm}\\
&\;\;\;\;\sum_{j \in J'} x_j + x_{i'} + |J'| x_{i''} \leq |J'|.
\end{align}
\end{subequations}
Furthermore, all inequalities in~\eqref{eq:V0_cdreduce} are dominated by linear combinations of the inequalities corresponding to $|J'|=1$, i.e., 
\begin{equation} \label{eq:base}
x_{j'} + x_{i'} + x_{i''} \leq 1,~\forall j' \in J \backslash (V_0 \cup \{\tilde{j}\}),~\forall \{i',i''\} \subseteq I \backslash V_0,
\end{equation}
and~\eqref{eq:base} is, in turn, dominated by~\eqref{eq:Iitself}. This implies that
$$
\begin{aligned}
\mathscr{P}_{(x_{V_0} = 0, x_{\tilde{j}}=1)} = \{ x \in \mathbb{R}^n \: | \: &x_{\tilde{j}}=1;~x_v = 0,~\forall v \in V_0; \\
& 0 \leq x_v \leq 1,~\forall v \in V \backslash (V_0 \cup \{\tilde{j}\});\\
&\sum_{i \in I} x_i + x_{j'} \leq 1,~\forall j' \in J \backslash (V_0 \cup \{\tilde{j}\})\}.
\end{aligned}
$$
Let $i'$ and $i''$ be two fixed vertices in $I$, and recall that $0 < \hat{x}_{i'},\hat{x}_{i''} < 1$. Correspondingly, let $\hat{x}^+$ and $\hat{x}^-$ be the two distinct points defined as follows for a small-enough value of $\epsilon > 0$:
$$
\begin{aligned}
&\hat{x}^+_{i'} = \hat{x}_{i'} + \epsilon;~\hat{x}^+_{i''} = \hat{x}_{i''} - \epsilon;~\hat{x}^+_v = \hat{x}_v, \forall v \in V \backslash \{i',i''\},\\
&\hat{x}^-_{i'} = \hat{x}_{i'} - \epsilon;~\hat{x}^-_{i''} = \hat{x}_{i''} + \epsilon;~\hat{x}^-_v = \hat{x}_v, \forall v \in V \backslash \{i',i''\}.
\end{aligned}
$$
Then, $\hat{x} = \frac{1}{2}(\hat{x}^+ + \hat{x}^-)$ while $\hat{x}^+,\hat{x}^- \in \mathscr{P}_{(x_{V_0} = 0, x_{\tilde{j}}=1)}$, which contradicts extremity of $\hat{x}$.

\item \textsc{case 3}: $|V_1| \geq 2$\\
The set of inequalities~\eqref{eq:V0} indicate that if $I \cap V_1 \neq \emptyset$ and $J \cap V_1 \neq \emptyset$, then each one of $I$ and $J$ must have exactly one vertex in $V_1$. Suppose $V_1 = \{i',j'\}$, for some $i' \in I$ and $j' \in J$. Then,~\eqref{eq:V0} immediately implies $\hat{x}_v =0 , \forall v \in V \backslash (V_0 \cup V_1)$, which contradicts $\hat{x}$ being fractional. On the other hand, if $I \cap V_1 = \emptyset$, then~\eqref{eq:V0_a}-\eqref{eq:V0_b} imply $\hat{x}_i = 0, \forall i \in I$, and~\eqref{eq:V0_c}-\eqref{eq:V0_d} become trivial and redundant. As a result, 
$$
\begin{aligned}
\mathscr{P}_{(x_{V_0} = 0, x_{V_1} = 1)} = \{ x \in \mathbb{R}^n \: | \:&x_v = 1, \forall v \in V_1;\\
&x_v = 0,~\forall v \in V_0; \\
&0 \leq x_v \leq 1,~\forall v \in V \backslash (V_0 \cup V_1) \},
\end{aligned}
$$
all extreme points of which are integral. Patently, the same result holds when $J \cap V_1 = \emptyset$, which contradicts fractionality of $\hat{x}$.
\end{itemize}
Hence, such a fractional extreme point $\hat{x}$ does not exist, and the proof is complete.
\end{proof}

It is clear that IUC, $1$-dependent set, and co-1-defective clique are equivalent in complete bipartite graphs. This leads to the following corollary concerning the $1$-dependent set and co-1-defective clique polytopes.
\begin{corollary}
Let $G=K_{|I|,|J|}$ be a complete bipartite graph, where $|I|,|J| \geq 2$. Then, the family of inequalities~\eqref{eq:doublestar} for all double-star subgraphs of $G$, together with the variable bounds, is sufficient to describe the $1$-dependent set (equivalently, co-1-defective clique) polytope associated with $G$.
\end{corollary}
\subsection{Fan and wheel}
\label{sec:fanwheel}
\noindent
A {\em fan} graph on $n$ vertices is composed of one vertex adjacent to $n-1$ vertices inducing a path (chain) subgraph. Following our notation in the previous section, we denote the vertex set of a fan graph by $F=\{h\} \cup P$, where $h$ is the hub vertex and $P$ denotes the vertex set of the corresponding path subgraph. Recall that a path graph itself is not facet-producing for the IUC polytope by Theorem~\ref{th:tree}. Conventionally, we assume that the edges of $G[P]$ are given by $\{i,i+1\}, \forall i \in \{1,\ldots,|P|-1\}$.
\begin{theorem} \label{th:fan} \textbf{\em (Fan Inequality)}
Let $F= \{h\} \cup P,~|P| \geq 4$, be a subset of vertices inducing a fan graph in $G=(V,E)$. Given $|P|=3q+r$, where $q$ is a positive integer and $r \in \{0,1,2\}$, the inequality
\begin{equation} \label{eq:fan}
\sum_{i \in P} x_i + \left ( 2(q-1) + \left \lfloor \frac{2(r+1)}{3} \right \rfloor \right ) x_h \leq 2q + \left \lfloor \frac{2(r+1)}{3} \right \rfloor
\end{equation}
is valid for $\mathscr{P}_{_{\!\!\mathcal{IUC}}}(G)$, and induces a facet of $\mathscr{P}_{_{\!\!\mathcal{IUC}}}(G[F])$ if and only if $r \neq 2$. 
\end{theorem}
\begin{proof}
First, note that the right-hand side of~\eqref{eq:fan} is the cardinality of a maximum IUC in $G[P]$; hence, the inequality
\begin{equation} \label{eq:Pvalid}
\sum_{i \in P} x_i \leq 2q + \left \lfloor \frac{2(r+1)}{3} \right \rfloor
\end{equation}
is valid for $\mathscr{P}_{_{\!\!\mathcal{IUC}}}(G)$. More formally,~\eqref{eq:Pvalid} is the Gomory-Chv\'{a}tal cut corresponding to the inequality obtained from the summation of all OT inequalities of $G[P]$ and $2 x_1 + x_2 + x_{|P|-1} + 2 x_{|P|} \leq 6$ divided by 3. Besides, every maximal IUC containing $h$ in $G[F]$ is a clique of size 3, since every pair of non-adjacent vertices in $G[P]$ with $h$ induces an open triangle in this subgraph. Therefore, $\max \{ \sum_{i \in P} x_i \: | \: x \in \mathscr{P}_{_{\!\!\mathcal{IUC}}}(G[F]) \text{ and } x_h = 1\} = 2$, and the coefficient of $x_h$ due to lifting into~\eqref{eq:Pvalid} is $2q + \big \lfloor \frac{2(r+1)}{3} \big \rfloor-2$. This establishes validity of~\eqref{eq:fan}.

Next, we show that $r \neq 2$ is sufficient for~\eqref{eq:fan} to induce a facet of the IUC polytope associated with $G[F]$. For $r=0$, consider the proper face $\mathcal{F} = \{x \in \mathscr{P}_{_{\!\!\mathcal{IUC}}}(G[F]) \: | \: \sum_{i \in P} x_i + 2(q-1) x_h = 2q \}$ of $\mathscr{P}_{_{\!\!\mathcal{IUC}}}(G[F])$ and an arbitrary supporting hyperplane $\mathscr{H}: a_h x_h + \sum_{i \in P} a_i x_i = b$ that contains $\mathcal{F}$. Let $x_{U_1}$ and $x_{U_2}$ be the incidence vectors of two maximum IUCs in $G[P]$ given as follows:
\begin{equation}
\begin{aligned}
U_1 &= \{1,3,3p-1,3p,\forall p \in \{2,\ldots,q\} \},\\
U_2 &= \{2,3,3p-1,3p,\forall p \in \{2,\ldots,q\} \}.
\end{aligned}
\end{equation}
These are the sets of black vertices in the path graphs obtained from eliminating the vertex in the last partition in the graphs of Figure~\ref{fig:hole_1}. Clearly, $x_{U_1},x_{U_2} \in \mathcal{F} \subseteq \mathscr{H}$, which immediately leads to $a_1 = a_2$. In addition, $a_{i}=a_{i+2}, \forall i \in \{1,\ldots,|P|-2\}$, as the incidence vector of every clique $\{h,i,i+1\},\forall i \in \{1,\ldots,|P|-1\}$, in $G[F]$ belongs to $\mathcal{F} \subseteq \mathscr{H}$. This indicates that $a_i = a, \forall i \in P$, $b=2aq$, and $a_h = 2a(q-1)$. Finally, $\mathscr{P}_{_{\!\!\mathcal{IUC}}}(G[F])$ having non-empty interior implies $a \neq 0$, so $\mathcal{F}$ is a facet of $\mathscr{P}_{_{\!\!\mathcal{IUC}}}(G[F])$. For the case $r=1$, using the same argument based on
$$
\begin{aligned}
U_1 &= \{3q+1,1,2,3p-2, 3p,\forall p \in \{2,\ldots,q\}\},\\
U_2 &= \{3q+1,1,3,3p-2, 3p,\forall p \in \{2,\ldots,q\}\},
\end{aligned}
$$
we may show that the inequality $\sum_{i \in P} x_i +  (2q-1) x_h \leq 2q+1$ is facet-inducing for $\mathscr{P}_{_{\!\!\mathcal{IUC}}}(G[F])$. This completes the sufficiency part of the proof.

It remains to prove that~\eqref{eq:fan} is not facet-defining for $\mathscr{P}_{_{\!\!\mathcal{IUC}}}(G[F])$ if $r = 2$. Observe that $G[F]$ has exactly $|P|-1$ IUCs containing $h$ whose 
incidence vectors belong to $\mathcal{F}= \{x \in \mathscr{P}_{_{\!\!\mathcal{IUC}}}(G[F]) \: | \: \sum_{i \in P} x_i + 2q x_h = 2q + 2 \}$, that is $\{h,i,i+1\},\forall i \in \{1,\ldots,|P|-1\}$. Thus, every other extreme point of $\mathscr{P}_{_{\!\!\mathcal{IUC}}}(G[F])$ that lies on $\mathcal{F}$ must be the incidence vector of a maximum IUC in $G[P]$. On the other hand, if $r=2$, $G[P]$ has a unique maximum IUC, whose incidence vector satisfies $x_1 = x_2 = x_{|P|-1} = x_{|P|} = 1$. Therefore, exactly $|P|$ extreme points of $\mathscr{P}_{_{\!\!\mathcal{IUC}}}(G[F])$ belong to $\mathcal{F}$, which implies \textit{dim}$(\mathcal{F}) \leq |P|-1 = |F|-2$. Recall that, every point $x \in \mathcal{F}$ is a convex combination of these extreme points. Hence, $\mathcal{F}$ is not a facet of $\mathscr{P}_{_{\!\!\mathcal{IUC}}}(G[F])$, and the proof is complete.
\end{proof}

Every fan graph contains quadratically many induced subgraphs of the same type, most of which satisfy the conditions of Theorem~\ref{th:fan}. The following theorem states that the fan inequalities corresponding to those subgraphs are facet-inducing for the IUC polytope associated with the supergraph.
\begin{theorem} \label{th:subfan}
Let $G=(F,E),~F= \{h\} \cup P$, be a fan graph. Then, every fan inequality~\eqref{eq:fan} corresponding to $F' = \{h\} \cup P',~P' \subseteq P$, satisfying the cardinality condition of Theorem~\ref{th:fan} induces a facet of $\mathscr{P}_{_{\!\!\mathcal{IUC}}}(G)$. 
\end{theorem}
\begin{proof}
Without loss of generality, let $P' = \{j,j+1,\ldots,k-1,k\}$ and consider lifting a variable $x_i, i \in P \backslash P'$ into the fan inequality~\eqref{eq:fan} corresponding to $F' = \{h\} \cup P'$. Note that $P' \cap N(i)= \emptyset, \forall i \in P \backslash \{j-1,\ldots,k+1\}$, thus every maximum IUC in $G[P']$ together with $i \in P \backslash \{j-1,\ldots,k+1\}$ forms an IUC in $G$. Besides, observe in the proof of Theorem~\ref{th:fan} that $G[P']$ always has a maximum IUC in which $j$ is an isolated vertex. As $P' \cap N(j-1)=\{j\}$, the union of this set and the vertex $j-1$ is also an IUC in $G$. By symmetry, the same result holds for the vertex $k+1$. This implies that the coefficient of every variable $x_i, \forall i \in P \backslash P'$, vanishes in an arbitrary sequential lifting of the fan inequality corresponding to $F'$, which completes the proof. 
\end{proof}

Next, we present similar results concerning the IUC polytope associated with a {\em wheel} graph. A wheel graph induced by $W=\{h\} \cup H$ consists of a hub vertex $h$ connected to all vertices of a hole $H$. As before, we assume that adjacent vertices in $G[H]$ have consecutive labels, and $|H|+1 \equiv 1$. Novelty of the following result mainly concerns the case that $|H|$ is a multiple of 3.
\begin{theorem} \label{th:wheel} \textbf{\em (Wheel Inequality)}
Let $W= \{h\} \cup H,~|H| \geq 4$, be a subset of vertices inducing a wheel graph in $G=(V,E)$. Given $|H|=3q+r$, where $q$ is a positive integer and $r \in \{0,1,2\}$, the inequality
\begin{equation} \label{eq:wheel}
\sum_{i \in H} x_i + \left ( 2(q-1) + \left \lfloor \frac{2r}{3} \right \rfloor \right ) x_h \leq 2q + \left \lfloor \frac{2r}{3} \right \rfloor
\end{equation}
is valid for $\mathscr{P}_{_{\!\!\mathcal{IUC}}}(G)$, and induces a facet of $\mathscr{P}_{_{\!\!\mathcal{IUC}}}(G[W])$ if and only if $r \neq 0$ or $q$ is odd. 
\end{theorem}
\begin{proof}
Validity of~\eqref{eq:wheel} as well as its facial property when $r \neq 0$ are evident, as~\eqref{eq:wheel} is obtained from lifting $x_h$ into the hole inequality~\eqref{eq:hole}.
Similar to a fan graph, $\max \{ \sum_{i \in H} x_i \: | \: x \in \mathscr{P}_{_{\!\!\mathcal{IUC}}}(G[W]) \text{ and } x_h = 1\} = 2$, hence the coefficient of $x_h$ in such a lifted hole inequality is $2q + \big \lfloor \frac{2r}{3} \big \rfloor-2$.

Let $r=0$, and consider $\mathcal{F} = \{x \in \mathscr{P}_{_{\!\!\mathcal{IUC}}}(G[W]) \: | \: \sum_{i \in H} x_i + 2(q-1) x_h = 2q \}$ as well as a supporting hyperplane $\mathscr{H}: a_h x_h + \sum_{i \in H} a_i x_i = b$ of $\mathscr{P}_{_{\!\!\mathcal{IUC}}}(G[W])$ containing it. The incidence vector of every clique $\{h,i,i+1\},\forall i \in \{1,\ldots,|H|\}$, in $G[W]$ belongs to $\mathcal{F} \subseteq \mathscr{H}$, thus $a_{i}=a_{i+2}, \forall i \in \{1,\ldots,|H|\}$. If $q$ is odd---equivalently, $H$ is of odd cardinality---this result further implies that $a_{|H|}=a_{2}$, which immediately leads to $a_{i}=a, \forall i \in \{1,\ldots,|H|\}$, $b=2aq$, and $a_h=2a(q-1)$. Since the interior of $\mathscr{P}_{_{\!\!\mathcal{IUC}}}(G[W])$ is non-empty, $a \neq 0$ and the sufficiency part of the proof is complete.

To prove the necessity, we examine the case where $r=0$ and $q$ is even. Similar to the proof of Theorem~\ref{th:fan}, our argument is based on the number of extreme points of $\mathscr{P}_{_{\!\!\mathcal{IUC}}}(G[W])$ that lie on $\mathcal{F}$, and the dimension of the affine subspace containing them. Patently, $G[W]$ has exactly $|H|$ IUCs containing $h$ whose incidence vectors belong to $\mathcal{F}$, i.e., $\{h,i,i+1\},\forall i \in \{1,\ldots,|H|\}$. Recall from Theorem~\ref{th:hole} that the right-hand-side of~\eqref{eq:wheel} is the IUC number of $G[H]$, hence every other extreme point of $\mathscr{P}_{_{\!\!\mathcal{IUC}}}(G[W])$ lying on $\mathcal{F}$ must be the incidence vector of a maximum IUC in $G[H]$. It is easy to verify that, if $|H|$ is a multiple of 3, $G[H]$ always has three distinct maximum IUCs (of cardinality $2q$) as follows: 
$$
\begin{aligned}
U_1 &= \{3p-2, 3p-1,\forall p \in \{1,\ldots,q\}\},\\
U_2 &= \{3p-1, 3p,\forall p \in \{1,\ldots,q\}\},\\
U_3 &= \{3p-2, 3p,\forall p \in \{1,\ldots,q\}\}.
\end{aligned}
$$
Hence, exactly $|H|+3$ extreme points of $\mathscr{P}_{_{\!\!\mathcal{IUC}}}(G[W])$ belong to $\mathcal{F}$. 

It is known that the adjacency matrix of a cycle graph with even number of vertices is rank deficient, which immediately implies that the affine subspace of $\mathbb{R}^{|H|}$ spanned by the points $x_{\{i,i+1\}} = e_i + e_{i+1}, \forall i \in \{1,\ldots,|H|\}$, is at most $|H|-2$ dimensional. Besides, the incidence vectors of $U_1,U_2,$ and $U_3$ in $\mathbb{R}^{|H|}$ can be easily written as linear combinations of these points, hence the dimension of the subspace of $\mathbb{R}^{|H|}$ spanned by $\{x_{U_1},x_{U_2},x_{U_3},x_{\{i,i+1\}}, \forall i \in \{1,\ldots,|H|\}\}$, is no more than $|H|-2$. Since all extreme points of $\mathscr{P}_{_{\!\!\mathcal{IUC}}}(G[W])$ lying on $\mathcal{F}$ are generated from a point in this set by adding the coordinate corresponding to $h$, the subspace of $\mathbb{R}^{|W|}$ spanned by them is at most $|H|-1 = |W|-2$ dimensional. Finally, because every point of $\mathcal{F}$ is a convex combination of these extreme points, dimension of this face is no more than $|W|-2$, therefore $\mathcal{F}$ is not a facet. The proof is complete.
\end{proof}

The result of Theorem~\ref{th:subfan} naturally extends to the induced fans of a wheel graph. The only exception is given by the subgraphs obtained from eliminating a single vertex (and incident edges) from the corresponding chordless cycle. Let $F=\{h\} \cup P$, where $P = H \backslash \{v\}$, be the vertex set of such a fan subgraph, while $|P|$ satisfies the facet-defining conditions of Theorem~\ref{th:fan}, that is $|P| = 3q+1$ or $|P| = 3q$. Observe that, in this case, $|H|$ also satisfies the facet-defining conditions of Theorem~\ref{th:wheel}, that is $|H| = 3q+2$ or $|H| = 3q+1$. Then, it can be easily shown that the wheel inequality~\eqref{eq:wheel} corresponding to $W=\{h\} \cup H$ coincides with the lifted fan equality~\eqref{eq:fan} corresponding to $F=\{h\} \cup P$. 

Patently, this result also applies to a {\em defective wheel} graph. A defective wheel is obtained from a (complete) wheel graph by eliminating some edges connecting the hub vertex to the corresponding hole. In addition to its fan subgraphs, a defective wheel always has at least one chordless cycle that contains its hub vertex. Examples of defective wheels are given in Figure~\ref{fig:holelift}. Let $H_h$ denote a hole containing the hub vertex $h$ in a defective wheel graph. In reference to Theorem~\ref{th:holelift} of Section~\ref{sec:cyclelift}, observe that every vertex not in $H_h$ has at most two neighbors in this set. The following corollary summarizes the implications of Theorems~\ref{th:subfan} and~\ref{th:holelift} regarding the facial structure of the IUC polytope associated with a defective wheel graph.
\begin{corollary}
Let $G=(V,E)$ be a defective wheel graph. Then, 
\begin{itemize}
\item[{\em (i)}] inequality~\eqref{eq:fan} corresponding to a fan subgraph satisfying the cardinality condition of Theorem~\ref{th:fan} is facet-defining for $\mathscr{P}_{_{\!\!\mathcal{IUC}}}(G)$. Besides,
\item[{\em (ii)}] inequality~\eqref{eq:hole} corresponding to a hole containing the hub vertex that satisfies the cardinality condition of Theorem~\ref{th:hole} induces a facet of $\mathscr{P}_{_{\!\!\mathcal{IUC}}}(G)$.
\end{itemize}
\end{corollary}

Finally, note that the IUC polytope associated with a fan graph, likewise a wheel graph, possesses exponentially many facets induced by the star inequality~\eqref{eq:star} as well.
\subsubsection*{Lifting fan and wheel inequalities}
\label{sec:fanwheellift}
\noindent
We consider lifting the fan and wheel inequalities under the conditions that they are facet-defining for the IUC polytope associated with the corresponding induced subgraph. Similar to Section~\ref{sec:cyclelift}, our focus is on the distribution of $N(v), v \in V \backslash F$ (resp. $v \in V \backslash W$), within $G[F]$ (resp. $G[W]$) and the cases where the lifting coefficients can be determined without solving a computationally expensive lifting problem.

First, observe that the result of Theorem~\ref{th:holelift}, concerning the hole inequality~\eqref{eq:hole}, also applies to the wheel inequality~\eqref{eq:wheel} since an IUC of cardinality $2q + \left \lfloor \frac{2r}{3} \right \rfloor + 1$ in $G[W \cup \{v\}], v \in V \backslash W,$ can always be attained through the vertices of $H$ (together with $v$) under the theorem conditions. However, the result concerning redundancy of the distribution of neighbors of $v$ does not extend to the fan inequality~\eqref{eq:fan} as well as the wheel inequality~\eqref{eq:wheel} when $r=0$ (and $q$ is odd). Consider, for example, the case where $|P|=7$ and $P \cap N(v) = \{1,3\}$, or $|H|=9$ and $H \cap N(v) = \{1,5\}$. In addition, it is trivial that $|P \cap N(v)| \leq 1$ (resp. $|H \cap N(v)| \leq 1$) is sufficient for the coefficient of $x_v, v \in V \backslash F$ (resp. $v \in V \backslash W$), to vanish in the latter cases, regardless of the role of the hub vertex $h$. In general, however, the lifting coefficient of a variable $x_v$ is highly influenced by connectivity of $v$ and $h$. Particularly, if $\{h,v\} \in E$, existence of two adjacent vertices in $P \cap N(v)$ (resp. $H \cap N(v)$) is sufficient for the corresponding lifting coefficient to vanish. On the other hand, if $\{h,v\} \notin E$, existence of two adjacent vertices in $P \backslash N(v)$ (resp. $H \backslash N(v)$) brings about the same result. The following theorem concerns redundancy of the distribution of $N(v)$ within $G[F]$ given the contribution of the hub vertex $h$. Note that, the coefficient of $x_v$ in every sequential lifting of the corresponding inequality will not exceed 2 because $\{h,v\}$ is always an IUC in the graph. 
\begin{theorem} \label{th:fanwheellift}
Consider the fan inequality~\eqref{eq:fan} where $r \neq 2$. In every sequential lifting of this inequality, the coefficient of a variable $x_v,~v \in V \backslash F$, vanishes under the following conditions: {\em (i)} $h \in N(v)$ and $|P \cap N(v)| \geq \left \lceil \frac{|P|}{2} \right \rceil + 1$, or {\em (ii)} $h \notin N(v)$ and $|P \cap N(v)| \leq \left \lfloor \frac{|P|}{2} \right \rfloor - 1$.
\end{theorem}
\begin{proof}
It suffices to observe that the cardinality of $P \cap N(v)$ in each case guarantees the existence of two adjacent vertices $i, i+1 \in P$ such that $\{h,v,i,i+1\}$ is an IUC in $G[F \cup \{v\}]$. In the former case, this IUC is a clique, and in the latter, it is the union of a clique and an isolated vertex.
\end{proof}

A similar result holds for the wheel inequality, which we skip here. 
We finish this section with presenting a family of facet-defining inequalities for the IUC polytope associated with an anti-cycle graph due to its fan substructures. We assume the vertices of an anti-cycle graph are labeled as described in the proof of Theorem~\ref{th:antihole}, and $|A|+1 \equiv 1$. 
\begin{theorem} \label{th:fanwheelliftantihole}
Let $G=(A,E)$ be an anti-cycle graph on $|A| \geq 8$ vertices. Then, the inequality
\begin{equation}
\sum_{j=i}^{i+3} x_j + x_k + x_{k+1} \leq 3,~\forall i \in A,~\forall k \in \{i+4, \ldots, |A|+i-2\}, 
\end{equation}
induces a facet of $\mathscr{P}_{_{\!\!\mathcal{IUC}}}(G)$.
\end{theorem}
\begin{proof}
Observe that $P = \{i,i+1,i+2,i+3\}, \forall i \in A$, induces a path (chain) graph on 4 vertices, and $G[\{k\} \cup P], \forall k \in \{i+5, \ldots, |A|+i-2\}$, is a fan in $G$. In fact, for a fixed vertex $i$, the vertices $i+4$ and $|A|+i-1$ are the only ones in $A \backslash P$ which are not adjacent to all vertices of $P$. Therefore, the inequality $\sum_{j=i}^{i+3} x_j + x_k \leq 3, \forall k \in \{i+5, \ldots, |A|+i-2\}$, is valid for $\mathscr{P}_{_{\!\!\mathcal{IUC}}}(G)$ and induces a facet of $\mathscr{P}_{_{\!\!\mathcal{IUC}}}(G[\{k\} \cup P])$, by Theorem~\ref{th:fan}. First, consider lifting the variable $x_{k+1}$ into $\sum_{j=i}^{i+3} x_j + x_{k} \leq 3$, for some $k \in \{i+5, \ldots, |A|+i-3\}$. Since $\{k,k+1\} \notin E$ and $P \cap N(k+1) = P$, the coefficient of $x_{k+1}$ in the corresponding lifted inequality will be 1 and the inequality $\sum_{j=i}^{i+3} x_j + x_{k}+x_{k+1} \leq 3$ is facet defining for $\mathscr{P}_{_{\!\!\mathcal{IUC}}}(G[\{k,k+1\} \cup P])$. Besides, every other vertex $v \in A \backslash (\{k,k+1\} \cup P)$ has at least one neighbor in $\{k,k+1\}$ and $|P \cap N(v)| \geq 3$, hence the coefficient of $x_v$ vanishes in every sequential lifting of this inequality, by Theorem~\ref{th:fanwheellift}. Therefore, the inequality $\sum_{j=i}^{i+3} x_j + x_{k}+x_{k+1} \leq 3$ induces a facet of $\mathscr{P}_{_{\!\!\mathcal{IUC}}}(G)$, for every $k \in \{i+5, \ldots, |A|+i-3\}$.
It remains to show that, for every $i \in A$, the inequalities $\sum_{j=i}^{i+3} x_j + x_{i+4}+x_{i+5} \leq 3$ and $\sum_{j=i}^{i+3} x_j + x_{|A|+i-2}+x_{|A|+i-1} \leq 3$ are facet-defining for $\mathscr{P}_{_{\!\!\mathcal{IUC}}}(G)$. We present the proof for the former inequality, i.e., 
\begin{equation} \label{eq:deffan}
\sum_{j=i}^{i+5} x_j \leq 3,
\end{equation}
for an arbitrary vertex $i$; the latter then follows by symmetry of the graph structure. It is easy to verify that~\eqref{eq:deffan} is obtained from lifting $x_{i+4}$ into the fan inequality $\sum_{j=i}^{i+3} x_j + x_{i+5} \leq 3$, hence facet-defining for the IUC polytope associated with the corresponding induced subgraph. Besides, every vertex $k \in \{i+7, \ldots, |A|+i-1\}$ is adjacent to the vertex $i+5$, thus the corresponding coefficient vanishes in lifting $x_k$ into~\eqref{eq:deffan}, by Theorem~\ref{th:fanwheellift}. Finally, observe that $\{i,i+2,i+4,i+6\}$ is a clique in $G$, therefore, $\max \{ \sum_{j=i}^{i+5} x_j \: | \: x \in \mathscr{P}_{_{\!\!\mathcal{IUC}}}(G) \text{ and } x_{i+6} = 1\} = 3$, and the coefficient of $x_{i+6}$ also vanishes in lifting this variable into~\eqref{eq:deffan}. The proof is complete.
\end{proof}

Note that the anti-cycle graph on $|A|=6$ vertices does not have an induced (complete) fan, and the inequality $\sum_{j=1}^{6} x_j \leq 3$ is not facet-inducing for the corresponding IUC polytope, previously shown by Theorem~\ref{th:antihole}. Also, for the anti-cycle graph on $|A|=7$ vertices, the inequalities $\sum_{j=1}^{6} x_j \leq 3$ and $\sum_{j=1}^{4} x_j + x_6 +x_7 \leq 3$ are both dominated by the anti-hole inequality $\sum_{j=1}^{7} x_j \leq 3$ of Theorem~\ref{th:antihole}.

\section{The Separation Problems}
\label{sec:separation}
\noindent
In this section, we study the computational complexity of the separation problem for each family of the valid inequalities presented in the previous section for the IUC polytope. 
Given a family of (linear) inequalities $\mathscr{F}$ and a point $x^*$, the corresponding separation problem is to decide whether $x^*$ belongs to the polyhedron defined by $\mathscr{F}$, i.e., if $x^*$ satisfies all inequalities in $\mathscr{F}$, and if not, find an inequality which is violated by $x^*$. Efficient separating subroutines are essential in devising branch-and-cut algorithms for integer (linear) programs; however, our results in this part are negative in that we can show that the separation problems for most of the proposed valid inequalities for the IUC polytope are NP-hard.

Complexity of the separation problem for the family of hole, as well as anti-hole, inequalities for the IUC polytope is yet to be determined, although there is a rich body of literature concerning identification of holes and anti-holes in graphs. In fact, the focus of this line of research has been mostly on holes and anti-holes of odd cardinality due to the Berge's strong perfect graph conjecture~\cite{Berge61WPGC}. According to this conjecture, a graph $G$ is perfect, i.e., the chromatic number of every induced subgraph of $G$ equals its clique number, if and only if $G$ does not contain an odd hole or odd anti-hole. Chudnovsky et al.~\cite{ChudnovskyA06} proved this conjecture, and Chudnovsky et al.~\cite{chudnovskyA2005} proposed an $\mathcal{O}(n^9)$-time algorithm to test if a graph has an odd hole or odd anti-hole; hence deciding if it is perfect. However, it remained an open problem whether or not it is polynomial-time decidable if a graph has an odd hole (equivalently, an odd anti-hole) until very recently that Chudnovsky et al.~\cite{chudnovskyA2019} presented an $\mathcal{O}(n^9)$-time algorithm for it. It is also known that the problem of deciding if a graph has a hole (equivalently, an anti-hole) of even cardinality is tractable~\cite{confortiA2002}, and the most efficient algorithm proposed for this problem has a time complexity of $\mathcal{O}(n^{11})$~\cite{changA2015}. On the other hand, it is NP-complete to decide if a graph has an odd (resp. even) hole containing a given vertex~\cite{bienstock1991,bienstock1992corrigendum,changA2015}. 
It is clear---due to feasibility of $x^*=\frac{2}{3}{\bf 1}$ to the fractional IUC polytope---that the separation problem for the family of hole inequalities~\eqref{eq:hole} for the IUC polytope is at least as hard as determining if a graph has a hole whose cardinality is not a multiple of 3, for which we are not aware of an efficient algorithm. Similarly, the feasibility of $x^*=\frac{1}{2}{\bf 1}$ to the fractional IUC polytope implies that the separation problem for the family of anti-hole inequalities~\eqref{eq:antihole} for the IUC polytope is at least as hard as deciding if a graph has an odd hole. It should be mentioned that the separation problem for the family of odd-hole inequalities for the vertex packing polytope (equivalently, odd-anti-hole inequalities for the clique polytope) is polynomial-time solvable~\cite{GLS}; however, this result does not automatically extend to the family of odd-anti-hole inequalities for the IUC polytope. 

In the rest of this section, we show that the separation problems for the families of star, double-star, fan and wheel inequalities for the IUC polytope are all NP-hard.
\begin{theorem} \label{th:starSeparation}
The separation problem for the family of star inequalities~\eqref{eq:star} for the IUC polytope is NP-hard.
\end{theorem}
\begin{proof}
The reduction is from the \textsc{Independent Set} problem ($k$-IS), which asks if a simple, undirected graph $G=(V,E)$ has an independent set of cardinality at least $k$. Without loss of generality, we assume $k \geq 3$. Corresponding to an instance $(G,k)$ of $k$-IS, consider an instance $(G',x^*)$ of the separation problem for the family of star inequalities~\eqref{eq:star}, where 
$G'=(V',E')$ is constructed by appending a new vertex $h$ to $G$ adjacent to all vertices of $V$, i.e.,
$$
\begin{aligned}
V'&=V \cup \{h\},\\
E'&=E \cup \{ \{h,i\},\forall i \in V \},
\end{aligned}
$$
and $x^*$ is defined as follows:
$$
\begin{aligned}
x^*_h &= 1, \\
x^*_i &= \frac{1}{k-1},~\forall i \in V.
\end{aligned}
$$
Note that $k \geq 3$ implies $x^* \in \mathscr{P}^{F}_{_{\!\!\mathcal{IUC}}}(G')$. We claim that $(G,k)$ is a positive instance of $k$-IS if and only if $x^*$ violates a star inequality in $G'$.
Let $I^*$ be an independent set of cardinality $k$ in $G$. Then, $S=\{h\} \cup I^*$ induces a star subgraph in $G'$, and 
$$
\sum_{i \in I^*} x^*_i + (|I^*|-1) x^*_h = \frac{k}{k-1}+(k-1)=k+\frac{1}{k-1} > k=|I^*|,
$$
which indicates violation of the star inequality associated with $G'[S]$. To prove the reverse, suppose that $x^*$ violates the inequality associated with a star subgraph induced by $S' = \{v\} \cup I'$ in $G'$, for some $I' \subseteq V'$ and $v \in V' \backslash I'$. Evidently, $I' \subseteq V$; since $x^*_v\leq 1=x^*_h, \forall v \in V'$, this also indicates violation of the star inequality associated with $G'[\{h\} \cup I']$, i.e., 
$$
\sum_{i \in I'} x^*_i + (|I'|-1) x^*_h = \frac{|I'|}{k-1} + (|I'|-1)  > |I'|,
$$
which implies $|I'| > k - 1$, or equivalently $|I'| \geq k$. The proof is complete.
\end{proof}
\begin{theorem} \label{th:DstarSeparation}
The separation problem for the family of double-star inequalities~\eqref{eq:doublestar} for the IUC polytope is NP-hard.
\end{theorem}
\begin{proof}
The proof is similar to Theorem~\ref{th:starSeparation}. Corresponding to an instance $(G,k)$ of $k$-IS, consider an instance $(G'=(V',E'),x^*)$ of the separation problem for the family of double-star inequalities~\eqref{eq:doublestar} constructed as follows:
$$
\begin{aligned}
&V'=V \cup \{h,u\},\\
&E'=E \cup \{ \{h,i\},\forall i \in V \} \cup \{ \{u,i\},\forall i \in V \},\\
& x^*_h = 1, \\
& x^*_i = \frac{1}{k},~\forall i \in V \cup \{u\}.
\end{aligned}
$$
Then, the same argument leads to the fact that $(G,k)$ is a positive instance of $k$-IS if and only if $x^*$ violates a double-star inequality in $G'$.
\end{proof}
\begin{theorem} \label{th:fanSeparation}
The separation problem for the family of fan inequalities~\eqref{eq:fan} for the IUC polytope is NP-hard.
\end{theorem}
\begin{proof}
We present a reduction from the \textsc{Induced Path} problem ($k$-IP), which is to decide if a simple, undirected graph $G=(V,E)$ has an induced path subgraph (equivalently, a chordless path) on at least $k$ vertices. Similar to $k$-IS, this problem is among the classical NP-complete problems listed by Garey and Johnson~\cite{GareyJohnson79}. The proof idea is similar to that for Theorem~\ref{th:starSeparation}. Let $k \geq 5$, and corresponding to an instance $(G,k)$ of $k$-IP, consider an instance $(G'=(V',E'),x^*)$ of the separation problem for the family of fan inequalities~\eqref{eq:fan} constructed as follows:
$$
\begin{aligned}
&V'=V \cup \{h\},\\
&E'=E \cup \{ \{h,i\},\forall i \in V \},\\
& x^*_h = 1, \\
& x^*_i = \frac{2}{k-1},~\forall i \in V.
\end{aligned}
$$
Let $P^*$ be the vertex set of a chordless path in $G$ on $k=3q_k+r_k$ vertices. Then, $F=\{h\} \cup P^*$ induces a fan subgraph in $G'$, and
$$
\begin{aligned}
\sum_{i \in P^*} x^*_i + \left ( 2(q_k-1) + \left \lfloor \frac{2(r_k+1)}{3} \right \rfloor \right ) x^*_h &= \frac{2k}{k-1} + 2q_k-2+\left \lfloor \frac{2(r_k+1)}{3} \right \rfloor\\
&=\frac{2}{k-1}+2q_k+\left \lfloor \frac{2(r_k+1)}{3} \right \rfloor\\
&> 2q_k + \left \lfloor \frac{2(r_k+1)}{3} \right \rfloor,
\end{aligned}
$$
which indicates violation of the fan inequality associated with $G'[F]$. Now, suppose that $x^*$ violates a fan inequality in $G'$. Given the magnitude of the components of $x^*$, this also implies violation of the inequality associated with a fan subgraph $G'[\{h\} \cup P']$, for some $P' \subseteq V$, i.e., 
$$
\sum_{i \in P'} x^*_i + \left ( 2(q'-1) + \left \lfloor \frac{2(r'+1)}{3} \right \rfloor \right ) x^*_h > 2q' + \left \lfloor \frac{2(r'+1)}{3} \right \rfloor,
$$
where $|P'|=3q'+r'$. This leads to

$$
\sum_{i \in P'} x^*_i = \frac{2|P'|}{k-1}>2,
$$
or equivalently, $|P'| \geq k$, and completes the proof.
\end{proof}
\begin{theorem} \label{th:wheelSeparation}
The separation problem for the family of wheel inequalities~\eqref{eq:wheel} for the IUC polytope is NP-hard.
\end{theorem}
\begin{proof}
The proof is almost identical to the proof of Theorem~\ref{th:fanSeparation}. The difference is that, instead of $k$-IP, the reduction is from the \textsc{Induced Cycle} problem, which is also known to be NP-complete~\cite{GareyJohnson79}. The \textsc{Induced Cycle} problem ($k$-IC) is to decide if a simple, undirected graph $G=(V,E)$ has a chordless cycle on at least $k$ vertices. Then, using the same argument on an instance $(G'=(V',E'),x^*)$ of $k$-IC constructed as stated in the proof of Theorem~\ref{th:fanSeparation}, one can show that the separation problem for the family of wheel inequalities~\eqref{eq:wheel} is NP-hard.
\end{proof}

It is worth noting that, from the viewpoint of parameterized complexity, the \textsc{Independent Set}, \textsc{Induced Path}, and \textsc{Induced Cycle} problems---parameterized by the solution size---are all W[1]-complete~\cite{chenA2007}, implying that it is unlikely that there exists a fixed-parameter-tractable (fpt) algorithm for either of them. This result immediately implies that the separation problems for the families of star, double-star, fan and wheel inequalities remain intractable even if the separation problem is defined with respect to a subset of inequalities from the corresponding family with a fixed (small) number of variables.  
\section{Computational Experiments}
\label{sec:experiment}
\noindent
In this section, we present the results of our computational experiments. We have conducted two sets of experiments to evaluate the effectiveness of incorporating the proposed valid inequalities in integer (linear) programming solution approaches for the {maximum IUC} problem.

\begin{table}[!b]
\centering{
\caption{First set of experiments: characteristics of the test instances.} \label{tab41}
\begin{tabular}{lrrrcccccc}
\toprule
 & & & & \multicolumn{6}{c}{Number of principal subgraphs} \\
\cmidrule(r){5-10} 
\multicolumn{1}{c}{{Name}} & \multicolumn{1}{c}{{$|V|$}} & \multicolumn{1}{c}{{$|E|$}} & \multicolumn{1}{c}{{$|\Lambda|$}} & \multicolumn{1}{c}{{Cycle}} & \multicolumn{1}{c}{{A-cycle}} & \multicolumn{1}{c}{{Star}} & \multicolumn{1}{c}{{D-star}} & \multicolumn{1}{c}{{Fan}} & \multicolumn{1}{c}{{Wheel}}\\
\hline
\texttt{rnd\#1} & 204 & 550 & 2,464 & 3 & 2 & 2 & 3 & 7 & 4 \\
\texttt{rnd\#2} & 213 & 630 & 2,857 & 2 & 4 & 6 & 1 & 3 & 5 \\
\texttt{rnd\#3} & 196 & 510 & 2,273 & 3 & 2 & 3 & 5 & 6 & 2 \\
\texttt{rnd\#4} & 194 & 470 & 2,056 & 6 & 2 & 3 & 3 & 2 & 5 \\
\texttt{rnd\#5} & 211 & 623 & 3,078 & 4 & 5 & 5 & 3 & 3 & 1 \\
\texttt{rnd\#6} & 194 & 617 & 2,673 & 2 & 8 & 2 & 4 & 2 & 3 \\
\texttt{rnd\#7} & 193 & 488 & 2,014 & 2 & 3 & 3 & 4 & 6 & 3 \\
\texttt{rnd\#8} & 198 & 506 & 2,226 & 5 & 2 & 4 & 4 & 4 & 2 \\
\texttt{rnd\#9} & 194 & 496 & 2,293 & 3 & 2 & 5 & 3 & 3 & 5 \\
\texttt{rnd\#10} & 194 & 472 & 1,998 & 2 & 3 & 5 & 6 & 3 & 2 \\
\texttt{rnd\#11} & 203 & 567 & 2,644 & 4 & 3 & 1 & 5 & 5 & 3 \\
\texttt{rnd\#12} & 188 & 466 & 2,143 & 2 & 1 & 2 & 5 & 8 & 3 \\
\texttt{rnd\#13} & 202 & 514 & 2,537 & 4 & 2 & 4 & 6 & 4 & 1 \\
\texttt{rnd\#14} & 212 & 497 & 2,070 & 7 & 2 & 5 & 2 & 3 & 2 \\
\texttt{rnd\#15} & 206 & 575 & 2,536 & 4 & 3 & 3 & 2 & 5 & 4 \\
\texttt{rnd\#16} & 201 & 550 & 2,608 & 2 & 3 & 2 & 7 & 2 & 5 \\
\texttt{rnd\#17} & 204 & 586 & 2,751 & 4 & 4 & 3 & 5 & 2 & 3 \\
\texttt{rnd\#18} & 200 & 581 & 2,813 & 3 & 5 & 4 & 3 & 2 & 4 \\
\texttt{rnd\#19} & 204 & 562 & 2,422 & 5 & 4 & 2 & 1 & 5 & 4 \\
\texttt{rnd\#20} & 203 & 550 & 2,268 & 5 & 3 & 2 & 3 & 7 & 1 \\
\bottomrule
\end{tabular}
}
\end{table}

In the first set of experiments, we investigated the effectiveness of each class of valid inequalities, as well as their combinations, using a set of graphs whose structures were known {\em a priori}. We generated a set of 20 undirected graphs, each of which is essentially a collection of 21 {\em principal} subgraphs that are sparsely connected to each other. Every principal subgraph of each instance has been randomly selected to be a cycle, anti-cycle, star, double-star, fan or wheel, and its number of vertices is given by the floor of a number drawn from a normal distribution with mean 10 and standard deviation 1. This choice of parameters yields graphs with around 200 vertices. The principal subgraphs have been connected to one another by a sparse set of random edges. To do this, an edge has been generated between every pair of vertices from two different principal subgraphs with a probability 0.01 independent of the others. Table~\ref{tab41} displays the characteristics of these instances, including the number of open triangles $|\Lambda|$ and the number of principal subgraphs for each structure type. In this table, ``A-cycle'' and ``D-star'' stand for anti-cycle and double-star, respectively.

We solved the {maximum IUC} problem on each instance using the base formulation~\eqref{eq:IPiuc}, i.e., considering only the OT inequalities, as well as using this formulation strengthened by the valid inequalities corresponding to its principal subgraphs. For each principal subgraph, we generated only one inequality, except for the double-star subgraphs where both inequalities of~\eqref{eq:doublestar} were generated. Recall that anti-cycle, star, double-star, fan and wheel graphs may generate several (potentially facet-defining) valid inequalities of the same or different types; yet, we restricted the incorporated inequalities to the whole structures, i.e., the principal subgraphs themselves, so we could have a better measure of the effectiveness of each family. We performed the experiments using ILOG/CPLEX 12.7 solver on a Dell Precision Workstation T7500$^{\tiny \circledR}$ machine with eight 2.40 GHz Intel Xeon$^{\tiny \circledR}$ processors and 12 GB RAM. We used the default settings of the solver except for the automatic cut generation, which was off in our first set of experiments.

\begin{table}[t]
\centering{
\caption{First set of experiments: improvement (in solution time) for each family of the valid inequalities.} \label{tab42}
\begin{tabular}{lrrrrrrrr}
\toprule
 & & \multicolumn{4}{c}{Solution time (CPU sec.)} & \multicolumn{3}{c}{Impv. per ineq. (\%)}\\
\cmidrule(r){3-6} \cmidrule(r){7-9} 
\multicolumn{1}{c}{{Name}} & \multicolumn{1}{c}{{opt.}} & \multicolumn{1}{c}{{Base}} & \multicolumn{1}{c}{{+HA}} & \multicolumn{1}{c}{{+SD}} & \multicolumn{1}{c}{{+FW}} & \multicolumn{1}{c}{{+HA}} & \multicolumn{1}{c}{{+SD}} & \multicolumn{1}{c}{{+FW}} \\
\hline
\texttt{rnd\#1} & 111 & 3,094.30  & 989.39   & 1,804.45 & 715.48   & {\bf 13.6} & 5.2  & 7.0 \\
\texttt{rnd\#2} & 118 & 1,103.08  & 765.27   & 791.47   & 369.09   & 5.1  & 3.5  & {\bf 8.3} \\
\texttt{rnd\#3} & 114 & 90.77     & 87.08    & 92.72    & 21.45    & 0.8  & $-$0.2 & {\bf 9.5} \\
\texttt{rnd\#4} & 110 & 197.38    & 180.75   & 177.79   & 108.90   & 1.1  & 1.1  & {\bf 6.4} \\
\texttt{rnd\#5} & 118 & 1,218.80  & 471.90   & 510.82   & 486.17   & 6.8  & 5.3  & {\bf 15.0} \\
\texttt{rnd\#6} & 103 & 9,820.76  & 250.19   & 5,012.75 & 2,299.58 & 9.7  & 4.9  & {\bf 15.3} \\
\texttt{rnd\#7} & 110 & 173.19    & 61.19    & 120.13   & 82.52    & {\bf 12.9} & 2.8  & 5.8 \\
\texttt{rnd\#8} & 115 & 115.95    & 54.36    & 75.58    & 87.51    & {\bf 7.6}  & 2.9  & 4.1 \\
\texttt{rnd\#9} & 113 & 95.45     & 83.50    & 51.37    & 22.74    & 2.5  & 4.2  & {\bf 9.5} \\
\texttt{rnd\#10} & 114 & 48.18     & 11.76    & 43.24    & 18.51    & {\bf 15.1} & 0.6  & 12.3 \\
\texttt{rnd\#11} & 114 & 438.31    & 197.32   & 471.63   & 182.44   & {\bf 7.9}  & $-$0.7 & 7.3 \\
\texttt{rnd\#12} & 110 & 73.44     & 20.98    & 53.08    & 13.16    & {\bf 23.8} & 2.3  & 7.5 \\
\texttt{rnd\#13} & 123 & 24.86     & 14.33    & 13.22    & 11.37    & 7.1  & 2.9  & {\bf 10.9} \\
\texttt{rnd\#14} & 124 & 156.53    & 89.51    & 84.94    & 90.49    & 4.8  & 5.1  & {\bf 8.4} \\
\texttt{rnd\#15} & 116 & 654.23    & 320.22   & 521.43   & 289.63   & {\bf 7.3}  & 2.9  & 6.2 \\
\texttt{rnd\#16} & 109 & 5,259.73  & 2,170.50 & 1,271.24 & 1,342.31 & {\bf 11.7} & 4.7  & 10.6 \\
\texttt{rnd\#17} & 111 & 3,870.85  & 1,096.95 & 944.77   & 1,126.95 & 9.0  & 5.8  & {\bf 14.2} \\
\texttt{rnd\#18} & 110 & 1,085.21  & 216.91   & 602.61   & 376.10   & 10.0 & 4.4  & {\bf 10.9} \\
\texttt{rnd\#19} & 110 & 10,496.40 & 1,971.21 & 3,847.48 & 708.52   & 9.0  & {\bf 15.8} & 10.4 \\
\texttt{rnd\#20} & 114 & 540.62    & 90.82    & 382.95   & 359.84   & {\bf 10.4} & 3.6  & 4.2 \\
\bottomrule
\end{tabular}
}
\end{table}

Table~\ref{tab42} presents the solution time for each instance using the base formulation~\eqref{eq:IPiuc} as well as the strengthened formulations by the individual classes of the proposed valid inequalities. In this table, ``opt.'' denotes the optimal solution value for each instance. The column ``Base'' reports the solution time under the base formulation~\eqref{eq:IPiuc}. The next three columns, named ``+HA'', ``+SD'', and ``+FW'', present the solution time for each instance when the base formulation is strengthened by the hole~\eqref{eq:hole} and anti-hole~\eqref{eq:antihole} inequalities, star~\eqref{eq:star} and double-star~\eqref{eq:doublestar} inequalities, and fan~\eqref{eq:fan} and wheel~\eqref{eq:wheel} inequalities associated with its principal subgraphs, respectively. The last three columns of this table show the percentage improvement in the solution time for each instance per appended valid inequality; that is, 
$$
\frac{\text{Base solution time}-\text{Improved solution time}}{\text{Base solution time} \times \text{Number of incorporated inequalities}} \times 100.$$
The boldface numbers in these columns show the highest improvement for each instance. Note that the highest per-inequality improvement does not necessarily imply the best solution time, due to different number of principal subgraphs of each type. On 10 instances (out of 20), the most per-inequality improvement was attained by the FW family (fan and wheel inequalities). With a similar performance, the HA family (hole and anti-hole inequalities) generated the highest per-inequality improvement on 9 instances. On the other hand, the SD family (star and double-star inequalities) showed an inferior performance compared to the two former families. In fact, on two instances {\tt rnd\#3} and {\tt rnd\#11}, incorporating these inequalities into the base formulation even resulted in a slight increase in the solution time. This observation may be justified by noting that star and double-star subgraphs generate exponentially many (with respect to the size of the subgraph) known facet-defining valid inequalities for the IUC polytope associated with the supergraph, whereas this number for the other two families is much smaller, and we have incorporated only a couple of these inequalities from each family. 

In the next phase of our experiments, we also examined the effect of adding different combinations of these valid inequalities on the solution time of the test instances. Table~\ref{tab43} presents the corresponding results. In this table, ``+HA\&SD'' shows the solution time when the base formulation has been strengthened by the combination of HA and SD families. Similarly, ``+HA\&FW'' and ``+SD\&FW'' show the results for the combination of HA and FW families, and SD and FW families, respectively. The last column of this table, i.e., ``+All'', presents the solution time for each instance obtained by incorporating all three HA, SD, and FW families in the base formulation~\eqref{eq:IPiuc}. Note that the total number of appended inequalities in this case is about 1\% of the number of OT inequalities for each instance. The boldface numbers in this table show the best solution time obtained for each instance. Even by incorporating a rather small number of valid inequalities, the solution times are substantially better than those obtained due to incorporating individual families of the valid inequalities, presented in Table~\ref{tab42}. Interestingly, for all instances, the best solution time was attained when the FW family was incorporated, and only on 11 instances (out of 20), the best solution time was due to incorporating all valid inequalities. 
\begin{table}[t]
\centering{
\caption{First set of experiments: results of incorporating different combinations of valid inequalities.} \label{tab43}
\begin{tabular}{lrrrrr}
\toprule
 & \multicolumn{5}{c}{Solution time (CPU sec.)} \\
\cmidrule(r){2-6} 
\multicolumn{1}{c}{{Name}} & \multicolumn{1}{c}{{Base}} & \multicolumn{1}{c}{{+HA\&SD}} & \multicolumn{1}{c}{{+HA\&FW}} & \multicolumn{1}{c}{{+SD\&FW}} & \multicolumn{1}{c}{{+All}} \\
\hline
\texttt{rnd\#1} & 3,094.30  & 477.81   & 345.94 & 291.17   & {\bf 167.73} \\
\texttt{rnd\#2} & 1,103.08  & 923.40   & 637.46 & {\bf 237.37}   & 354.18 \\
\texttt{rnd\#3} & 90.77     & 29.96    & 19.81  & 25.13    & {\bf 16.91}  \\
\texttt{rnd\#4} & 197.38    & 158.29   & 103.24 & {\bf 91.26}    & 103.08 \\
\texttt{rnd\#5} & 1,218.80  & 326.88   & 356.16 & 583.89   & {\bf 242.28} \\
\texttt{rnd\#6} & 9,820.76  & 274.15   & {\bf 119.29} & 1,346.66 & 157.43 \\
\texttt{rnd\#7} & 173.19    & 32.46    & 36.02  & 58.75    & {\bf 23.43}  \\
\texttt{rnd\#8} & 115.95    & 26.17    & {\bf 19.85}  & 84.22    & 25.78  \\
\texttt{rnd\#9} & 95.45     & 68.12    & 17.80  & 19.70    & {\bf 15.05}  \\
\texttt{rnd\#10} & 48.18     & 11.85    & 15.05  & 9.29     & {\bf 7.47}   \\
\texttt{rnd\#11} & 438.31    & 152.38   & {\bf 123.44} & 158.43   & 131.15 \\
\texttt{rnd\#12} & 73.44     & 28.09    & {\bf 10.14}  & 14.47    & 14.18  \\
\texttt{rnd\#13} & 24.86     & 7.28     & {\bf 1.99}   & 8.93     & 8.40   \\
\texttt{rnd\#14} & 156.53    & 97.76    & 50.99  & 87.31    & {\bf 47.58}  \\
\texttt{rnd\#15} & 654.23    & 287.86   & 147.35 & 216.37   & {\bf 103.77} \\
\texttt{rnd\#16} & 5,259.73  & 722.94   & 665.09 & 414.69   & {\bf 157.23} \\
\texttt{rnd\#17} & 3,870.85  & 846.85   & 635.40 & 510.96   & {\bf 332.18} \\
\texttt{rnd\#18} & 1,085.21  & 138.07   & 127.95 & 223.97   & {\bf 102.35} \\
\texttt{rnd\#19} & 10,496.40 & 1,391.08 & {\bf 283.62} & 546.51   & 354.59 \\
\texttt{rnd\#20} & 540.62    & 93.92    & {\bf 61.29}  & 191.92   & 68.03 \\
\bottomrule
\end{tabular}
}
\end{table}

The results of our first set of experiments reveal the merit of the proposed valid inequalities. In practice, however, the structure of an input graph is not known {\em a priori}. Therefore, integer (linear) programming solution methods of the {maximum IUC} problem should rely on subgraph-detection routines in order to generate these valid inequalities. In our second set of experiments, we investigated efficacy of the proposed valid inequalities under this setting. We generated a set of (Erd\"{o}s-R\'{e}nyi) random graphs $G(n,p)$ with $n=100$ and $p$ varying from $0.05$ to $0.95$. In these graphs, $n$ is the number of vertices and an edge exists between every pair of vertices with a probability $p$ independent of the others. The characteristics of our second set of test instances, including the number of open triangles $|\Lambda|$ for each instance, are shown in Table~\ref{tab44}. In this table, the numerical part of the name of each instance shows the corresponding value of $p$, which is also the expected edge density of the graph. In light of the results of our first set of experiments, as well as the complexity results of the previous section, we just considered 4-hole, wheel and fan cutting planes in this part of our experiments, and applied them all to the root node of the branch-and-cut tree. It should be mentioned that developing a full-blown branch-and-cut algorithm for the {maximum IUC} problem is out of the scope this paper, and the purpose of our experiments was to test effectiveness of employing these inequalities using a straightforward strategy, which is described below.
\begin{table}[t]
\centering{
\caption{Second set of experiments: cut-generation characteristics.} \label{tab44}
\begin{tabular}{lrrrrrrrr}
\toprule
 & & & & \multicolumn{4}{c}{\# Valid inequalities (cuts)} & \multicolumn{1}{c}{Cut  generation}\\
\cmidrule(r){5-8}
\multicolumn{1}{c}{{Name}} & \multicolumn{1}{c}{{$|V|$}} & \multicolumn{1}{c}{{$|E|$}} & \multicolumn{1}{c}{{$|\Lambda|$}} & \multicolumn{1}{c}{{OT}} & \multicolumn{1}{c}{{4-hole}} & \multicolumn{1}{c}{{Wheel}} & \multicolumn{1}{c}{{Fan}} & \multicolumn{1}{c}{{(CPU sec.)}} \\
\hline
\texttt{RND\_0.05} & 100 & 255   & 1,204  & 956   & 70     & 0     & 0     & 0.05   \\
\texttt{RND\_0.10} & 100 & 503   & 4,558  & 1,867 & 868    & 9     & 0     & 0.06   \\
\texttt{RND\_0.15} & 100 & 810   & 10,823 & 1,235 & 3,444  & 378   & 211   & 0.34   \\
\texttt{RND\_0.20} & 100 & 1,058 & 17,337 & 511   & 6,126  & 1,614 & 965   & 1.40   \\
\texttt{RND\_0.25} & 100 & 1,261 & 23,406 & 274   & 8,378  & 2,463 & 1,323 & 3.57   \\
\texttt{RND\_0.30} & 100 & 1,541 & 32,444 & 30    & 10,945 & 3,338 & 1,466 & 9.623   \\
\texttt{RND\_0.35} & 100 & 1,812 & 41,242 & 7     & 13,683 & 3,382 & 1,616 & 20.19  \\
\texttt{RND\_0.40} & 100 & 1,966 & 45,987 & 2     & 14,988 & 3,470 & 1,528 & 27.58  \\
\texttt{RND\_0.45} & 100 & 2,253 & 54,553 & 1     & 17,523 & 3,506 & 1,494 & 46.07  \\
\texttt{RND\_0.50} & 100 & 2,470 & 60,337 & 0     & 19,323 & 3,456 & 1,544 & 63.04   \\
\texttt{RND\_0.55} & 100 & 2,773 & 67,027 & 0     & 21,103 & 3,604 & 1,396 & 87.25  \\
\texttt{RND\_0.60} & 100 & 2,981 & 70,308 & 0     & 21,897 & 3,687 & 1,313 & 101.20 \\
\texttt{RND\_0.65} & 100 & 3,223 & 71,881 & 0     & 22,766 & 3,589 & 1,384 & 108.66 \\
\texttt{RND\_0.70} & 100 & 3,457 & 71,135 & 0     & 22,923 & 3,747 & 1,184 & 104.21 \\
\texttt{RND\_0.75} & 100 & 3,748 & 67,644 & 0     & 21,707 & 2,328 & 700   & 91.00  \\
\texttt{RND\_0.80} & 100 & 3,994 & 61,000 & 0     & 19,990 & 350   & 62    & 66.70  \\
\texttt{RND\_0.85} & 100 & 4,259 & 50,326 & 0     & 16,822 & 2     & 24    & 38.02  \\
\texttt{RND\_0.90} & 100 & 4,460 & 39,011 & 2     & 13,635 & 0     & 0     & 17.86  \\
\texttt{RND\_0.95} & 100 & 4,726 & 19,977 & 762   & 7,509  & 0     & 0     & 2.53  \\
\bottomrule
\end{tabular}
}
\end{table}

For each instance, we identified the entire set of 4-holes, but we employed a subset of them to generate the 4-hole cutting planes. Note that the expected number of 4-holes in random graphs with moderate densities is much higher than the expected number of open triangles. Besides, every 4-hole inequality~\eqref{eq:hole} dominates four OT inequalities, which can be eliminated from the original formulation upon adding the 4-hole cuts. In this regard, our solution strategy was to find a small subset of 4-hole inequalities that could cover a large portion (possibly all) of OT inequalities. Clearly, finding a minimum number of 4-holes to dominate all possible open triangles in a graph is a set covering problem, which is NP-hard in general. We employed a well-known greedy heuristic for this problem to generate the 4-hole cuts; at every iteration, we selected a 4-hole with the maximum number of uncovered open triangles, and stopped when this number was zero for all the remaining 4-holes. Obviously, not all open triangles are guaranteed to be part of a 4-hole in a graph, thus not all OT inequalities could be replaced with the generated 4-hole cuts. Table~\ref{tab44} shows the number of generated 4-hole cuts for each instance, as well as the number of OT inequalities that were not covered by them and remained in the problem formulation.

We also generated a set of wheel and fan cuts for each instance. To do this, for every vertex $i \in V$, we considered the subgraph induced by its neighbors, i.e., $G[N(i)]$, and detected a set of chordless cycles and paths using a simple enumeration method. Clearly, a graph may contain an exponential number of such structures, thus we restricted the search on each subgraph to at most 50 cycles and paths or 1.00 CPU second, whichever came first before the search was complete. In order to generate high-quality wheel and fan cuts, we only considered (chordless) cycles and paths with at least 7 vertices that satisfied the facet-defining conditions of Theorems~\ref{th:fan} and~\ref{th:wheel}. The number of generated wheel and fan cuts for each instance is presented in Table~\ref{tab44}. The last column of this table shows the aggregate time spent on generating all 4-hole, wheel and fan cuts for each instance. It should be mentioned that, for all instances, the CPU time taken to generate the wheel and fan inequalities was negligible compared to the required time to generate the 4-hole cuts. In fact, the maximum time spent by the algorithm to generate the entire set of wheel and fan cuts for an instance remained below 1.20 CPU seconds.

The results of the second set of experiments are presented in Table~\ref{tab45}. We solved each instance using the base formulation~\eqref{eq:IPiuc}, as well as its enhancement obtained by adding the generated valid inequalities, denoted by ``+VI'' in Table~\ref{tab45}. The second column of this table shows the optimal solution value of the {maximum IUC} problem for each instance. The next two columns present the optimal solution value of the corresponding LP relaxation problems. The last four columns of this table compare the size of the branch-and-cut (B\&C) trees and solution times. In this set of experiments, the automatic cut-generation of the solver was also on. The results clearly show the effectiveness of the proposed valid inequalities. Significant improvements in the solution times are observed for the graphs with the edge density of less than 65\%, as well as those with the edge density of 80\% and more, due to incorporating the generated cutting planes. The +VI solution times were slightly better than those for the base formulation on {\tt RND\_0.65} and {\tt RND\_0.75}, and slightly worse on {\tt RND\_0.70}. The better solution time for each instance is shown in bold in Table~\ref{tab45}. For all instances, the number of B\&C nodes was considerably smaller under the +VI strategy.

\begin{table}[t]
\centering{
\caption{Second set of experiments: integer (linear) programming solution results.} \label{tab45}
\begin{tabular}{lcccrrrr}
\toprule
 & & \multicolumn{2}{c}{LP rlx. opt. value} & \multicolumn{2}{c}{\# B\&C nodes} & \multicolumn{2}{c}{Solution time (CPU sec.)}\\
\cmidrule(r){3-4} \cmidrule(r){5-6} \cmidrule(r){7-8} 
\multicolumn{1}{c}{{Name}} & \multicolumn{1}{c}{{opt.}} & \multicolumn{1}{c}{{Base}} & \multicolumn{1}{c}{{+VI}} & \multicolumn{1}{c}{{Base}} & \multicolumn{1}{c}{{+VI}} & \multicolumn{1}{c}{{Base}} & \multicolumn{1}{c}{{+VI}} \\
\hline
\texttt{RND\_0.05} & 53 & 66.67 & 59.22 & 3,241     & 0       & 22.51     & {\bf 5.63}     \\
\texttt{RND\_0.10} & 39 & 66.67 & 49.71 & 213,358   & 17,005  & 1,778.93  & {\bf218.46}   \\
\texttt{RND\_0.15} & 28 & 66.67 & 46.60 & 2,017,952 & 160,544 & 35,357.90 & {\bf2,284.37} \\
\texttt{RND\_0.20} & 24 & 66.67 & 44.92 & 1,112,661 & 180,755 & 23,173.20 & {\bf4,248.66} \\
\texttt{RND\_0.25} & 21 & 66.67 & 44.30 & 1,159,561 & 192,961 & 34,036.80 & {\bf5,123.73} \\
\texttt{RND\_0.30} & 18 & 66.67 & 44.29 & 724,553   & 227,952 & 35,861.80 & {\bf8,733.08} \\
\texttt{RND\_0.35} & 16 & 66.67 & 44.33 & 615,306   & 191,929 & 30,004.00 & {\bf8,444.93} \\
\texttt{RND\_0.40} & 15 & 66.67 & 44.44 & 468,742   & 125,469 & 26,178.80 & {\bf6,928.29} \\
\texttt{RND\_0.45} & 13 & 66.67 & 44.44 & 295,550   & 126,449 & 19,691.90 & {\bf7,250.91} \\
\texttt{RND\_0.50} & 12 & 66.67 & 44.44 & 222,235   & 104,277 & 18,260.00 & {\bf8,861.15} \\
\texttt{RND\_0.55} & 11 & 66.67 & 44.44 & 176,862   & 70,025  & 13,810.80 & {\bf6,621.82} \\
\texttt{RND\_0.60} & 11 & 66.67 & 44.44 & 111,914   & 49,889  & 9,102.53  & {\bf5,930.36} \\
\texttt{RND\_0.65} & 13 & 66.67 & 44.44 & 66,995    & 33,969  & 5,615.65  & {\bf4,927.96} \\
\texttt{RND\_0.70} & 15 & 66.67 & 44.44 & 57,423    & 31,844  & {\bf 5,273.13}  & 5,414.28 \\
\texttt{RND\_0.75} & 19 & 66.67 & 44.48 & 31,135    & 19,714  & 5,163.23  & {\bf4,490.54} \\
\texttt{RND\_0.80} & 21 & 66.67 & 46.63 & 74,719    & 29,461  & 7,943.90  & {\bf4,605.04} \\
\texttt{RND\_0.85} & 25 & 66.67 & 49.75 & 46,781    & 22,872  & 4,574.90  & {\bf2,404.52} \\
\texttt{RND\_0.90} & 31 & 66.67 & 50.00 & 18,562    & 7,787   & 1,897.80  & {\bf492.21}   \\
\texttt{RND\_0.95} & 46 & 66.67 & 51.50 & 0         & 0       & 34.24     & {\bf5.16}    \\
\bottomrule
\end{tabular}
}
\end{table}
\section{Conclusion}
\label{sec:conclusion}
In this paper, we presented a study of the IUC polytope associated with a simple undirected graph $G=(V,E)$, defined as the convex hull of the incidence vectors of all subsets of $V$ inducing cluster subgraphs of $G$. It was shown that the fractional IUC polytope, obtained from relaxing integrality of the variables in definition of the original polytope, is a polyhedral outer-approximation of a cubically constrained region in the space of the original variables; hence, it provides a very weak approximation of the (integral) IUC polytope, remarking the importance of incorporating strong valid inequalities in integer (linear) programming solution methods of the {\sc Maximum IUC} problem. We derived several facet-defining valid inequalities for the IUC polytope associated with cycle, anti-cycle, star, double-star, fan and wheel graphs, along with the conditions that they remain facet-defining for the IUC polytope associated with a supergraph containing them, as well as some results concerning the corresponding lifting procedures. We also presented a complete description of the IUC polytope for some special classes of graphs. In our presentation, we spotted the similarities between the facial structure of the IUC polytope and those of the independent set (vertex packing) and clique polytopes. For each family of the proposed valid inequalities, we studied the computational complexity of the corresponding separation problem. In particular, we showed that the separation problem for the families of star, double-star, fan and wheel inequalities are NP-hard. We also examined the effectiveness of the proposed valid inequalities when employed in an integer (linear) programming solution method of the {\sc Maximum IUC} problem through computational experiments. Our results showed great improvement in solution time of this problem, revealing the merit of the proposed valid inequalities from the computational standpoint.

As an extension of this research, one may search for other facet-producing structures for the IUC polytope. As stated previously, an immediate step could be to investigate webs that subsume cycle and anti-cycle graphs. Developing detailed branch-and-cut algorithms for the {\sc Maximum IUC} problem based on effective heuristic separation procedures for the proposed valid inequalities in this paper is another direction for future research. Another interesting direction is to explore LP-based scale reduction techniques for the {\sc Maximum IUC} problem. It should be mentioned that, because of the cubic nature of the IUC formulation~\eqref{eq:IPiuc} as explained before, $\mathscr{P}^{F}_{_{\!\!\mathcal{IUC}}}(G)$ does not generally hold the favorable structural properties of the fractional independent set (vertex packing) and clique polytopes such as persistency~\cite{NemhauserTrotter75}. However, we observed through a set of computational experiments that the frequency of non-persistent optimal solutions to LP relaxation of the IUC formulation~\eqref{eq:IPiuc} is not high. In this regard, it is of interest to identify the conditions under which an LP relaxation optimal solution of~\eqref{eq:IPiuc} shows persistency.
\paragraph*{Acknowledgements}
We would like to thank the three anonymous referees for their constructive comments and suggestions.
Partial support by AFOSR award FA9550-19-1-0161 is also gratefully acknowledged.

\end{document}